\documentclass[10pt]{amsart}
\renewcommand\appendix{\par
  \setcounter{section}{0}
  \setcounter{subsection}{0}
  \setcounter{equation}{0}
  \setcounter{figure}{0}
  \setcounter{table}{0}
  \renewcommand\thesection{Appendix \Alph{section}}
  \renewcommand\theequation{\Alph{section}\arabic{equation}}
  \renewcommand\thefigure{\Alph{section}\arabic{figure}}
  \renewcommand\thetable{\Alph{section}\arabic{table}}
}

\usepackage{graphicx}
\usepackage{color}

\pdfoutput=1
\newtheorem{theorem}{Theorem}[section]
\newtheorem{lemma}{Lemma}[section]

\newcommand{\tg}{\widetilde{g}}

\newtheorem{proposition}{Proposition}[section]
\newtheorem{corollary}{Corollary}[section]

\theoremstyle{definition}
\newtheorem{definition}{Definition}[section]
\newtheorem{example}{Example}[section]

\newtheorem{remark}{Remark}[section]

\renewcommand{\oe}{$m=\infty$ when $n$ is odd and $m=\frac{n-2}{2}$ when $n$ is even}
\usepackage{amsmath}
\usepackage{amsfonts}
\usepackage{amssymb}
\usepackage{mathrsfs}
\newcommand{\be}{\begin{equation}}
\newcommand{\ee}{\end{equation}}

\newcommand{\bnum}{\begin{enumerate}}
\newcommand{\enum}{\end{enumerate}}

\newcommand{\del}{\partial}

\newcommand{\go}{{\stackrel{\scriptscriptstyle{0}}{g}}\phantom{}}
\newcommand{\gtop}[1]{{\stackrel{\scriptscriptstyle{#1}}{h}}\phantom{}}
\newcommand{\htop}[1]{{\stackrel{\scriptscriptstyle{#1}}{h}}\phantom{}}
\newcommand{\ggo}{{\stackrel{\scriptscriptstyle{0}}{\mathbf{g}}}\phantom{}}

\newcommand{\no}{{\stackrel{\scriptscriptstyle{0}}{\nabla}}\phantom{}}
\newcommand{\Ro}{{\stackrel{\scriptscriptstyle{0}}{R}}\phantom{}}
\newcommand{\Lo}{{\stackrel{\scriptscriptstyle{0}}{\Delta}}\phantom{}}
\newcommand{\Bo}{{\stackrel{\scriptscriptstyle{0}}{\Box}}\phantom{}}
\newcommand{\h}{\mathbf{h}}
\renewcommand{\gg}{\mathbf{g}}

\newcommand{\f}{\varphi}

\newcommand{\ba}{\begin{array}}

\newcommand{\ea}{\end{array}}

\newcommand{\beq}{\begin{eqnarray}}
\newcommand{\beqn}{\begin{eqnarray*}}

\newcommand{\eeq}{\end{eqnarray}}
\newcommand{\eeqn}{\end{eqnarray*}}

\newtheorem{lm}{lemma}
\newcommand{\blem}{\begin{lemma}}
\newcommand{\elem}{\end{lemma}}

\newtheorem{thee}{theorem}

\newtheorem{proo}{proposition}

\newtheorem{co}{corollary}

\newtheorem{rem}{remark}

\newtheorem{deff}{definition}

\newcommand{\bd}{\begin{deff}}

\newcommand{\ed}{\end{deff}}

\newcommand{\bl}{\begin{lm}}

\newcommand{\el}{\end{lm}}

\newcommand{\bp}{\begin{proo}}

\newcommand{\ep}{\end{proo}}

\newcommand{\bt}{\begin{thee}}

\newcommand{\et}{\end{thee}}

\newcommand{\bc}{\begin{co}}

\newcommand{\ec}{\end{co}}

\newcommand{\brm}{\begin{rem}}

\newcommand{\erm}{\end{rem}}

\newcommand{\der}{{\rm d}}

\usepackage{t1enc}
\def\frak{\mathfrak}

\def\cal{\mathcal}

\newcommand{\newc}{\newcommand}

\newcommand{\Ker}{\operatorname{Ker}}
\renewcommand{\Im}{\operatorname{Im}}

\let\ccdot\cdot
\def\cdot{\hbox to 2.5pt{\hss$\ccdot$\hss}}

\newc{\aR}{\mbox{\boldmath{$ R$}}}
\newc{\aS}{\mbox{\boldmath{$ S$}}}
\newc{\aT}{\mbox{\boldmath{$ T$}}}
\newc{\aW}{\mbox{\boldmath{$ W$}}}

\newc{\aK}{\mbox{\boldmath{$ K$}}}
\newc{\aL}{\mbox{\boldmath{$ L$}}}

\usepackage{amssymb}
\usepackage{amscd}

\renewcommand{\P}{\mathsf{P}}

\newcommand{\Ric}{\operatorname{Ric}}

\let\e=\varepsilon
\let\f=\varphi
\let\i=\iota

\newcommand{\hook}{\makebox[7pt]{\rule{6pt}{.3pt}\rule{.3pt}{5pt}}\,}

\newc{\obstrn}[2]{B^{#1}_{#2}}

\newcommand{\rpl}                         
{\mbox{$
\begin{picture}(12.7,8)(-.5,-1)
\put(0,0.2){$+$}
\put(4.2,2.8){\oval(8,8)[r]}
\end{picture}$}}

\newcommand{\lpl}                         
{\mbox{$
\begin{picture}(12.7,8)(-.5,-1)
\put(2,0.2){$+$}
\put(6.2,2.8){\oval(8,8)[l]}
\end{picture}$}}

\renewcommand{\aa}{{\bar{a}}}
\newcommand{\bb}{{\bar{b}}}
\newcommand{\cc}{{\bar{c}}}
\newcommand{\dd}{{\bar{d}}}

\usepackage{ifthen}

\newcommand{\gla}{\frak{gl}}

\newc{\tensor}[1]{#1}
\newc{\Mvariable}[1]{\mbox{#1}}
\newc{\down}[1]{{}_{#1}}
\newc{\up}[1]{{}^{#1}}

\newcommand{\+}{\oplus}
\newcommand{\lag}{\mathfrak{g}}
\newcommand{\lah}{\mathfrak{h}}

\newcommand{\laz}{\mathfrak{z}}
\newcommand{\z}{\mathfrak{z}}

\newcommand{\lak}{\mathfrak{k}}
\newcommand{\lam}{\mathfrak{m}}

\newc{\JulyStrut}{\rule{0mm}{6mm}}
\newc{\midtenPan}{\mbox{\sf S}}
\newc{\midten}{\mbox{\sf T}}
\newc{\midtenEi}{\mbox{\sf U}}
\newc{\ATen}{\mbox{\sf E}}
\newc{\BTen}{\mbox{\sf F}}
\newc{\CTen}{\mbox{\sf G}}

\def\sideremark#1{\ifvmode\leavevmode\fi\vadjust{\vbox to0pt{\vss
 \hbox to 0pt{\hskip\hsize\hskip1em
 \vbox{\hsize3cm\tiny\raggedright\pretolerance10000
 \noindent #1\hfill}\hss}\vbox to8pt{\vfil}\vss}}}%
                        
\numberwithin{equation}{section}

\newcounter{romenumi}
\newcommand{\labelromenumi}{(\roman{romenumi})}

\newcommand{\bma}{\begin{pmatrix}}
\newcommand{\ema}{\end{pmatrix}}

\begin{document}
 \newcommand{\marg}[1]{\marginpar{\small {\tt #1}}}
\bibliographystyle{abbrv}
\title[Conformal Walker metrics and linear Fefferman-Graham equations]{Conformal Walker metrics and linear Fefferman-Graham equations} 
\vskip 1.truecm 
\author{Ian M Anderson} \address{Department of Mathematics and Statistics, Utah State University, Logan Utah, 84322, USA}
\email{\tt Ian.Anderson@usu.edu}
\author{Thomas Leistner}\address{School of Mathematical Sciences, University of Adelaide, SA 5005, Australia} \email{\tt thomas.leistner@adelaide.edu.au}
\author{Andree Lischewski}\address{Institut f\"{u}r Mathematik, Humboldt-University Berlin,
Unter den Linden 6,
10099 Berlin, Germany} \email{\tt  lischews@math.hu-berlin.de}
\author{Pawe\l~ Nurowski} 
\address{Centrum Fizyki Teoretycznej PAN, 
Al. Lotnik\'{o}w 32/46, 
02-668 Warszawa, Poland}
\email{nurowski@cft.edu.pl} \thanks{This research was supported by the Australian Research
Council via the grants FT110100429 and DP120104582 and by the
  Polish Ministry of Research and Higher Education under the grants NN201~607540 and NN202~104838.}
\date{\today}

\begin{abstract} The conformal Fefferman-Graham ambient metric construction is one of the most fundamental constructions in conformal geometry. It embeds a manifold with a conformal structure into a  pseudo-Riemannian manifold whose Ricci tensor vanishes up to a certain order along the original manifold. Despite the general existence result of such ambient metrics by Fefferman and Graham,  not many examples of conformal structures with Ricci-flat ambient metrics are known. Motivated by previous examples, 
for which the Fefferman-Graham equations for the ambient metric to be Ricci-flat reduce to a system of linear PDEs, in the present article we develop a method to find ambient metrics for conformal classes of metrics with two-step nilpotent Schouten tensor. Using this method, for metrics for which the image of the Schouten tensor is invariant under parallel transport, i.e., certain types of Walker metrics, we obtain explicit ambient metrics. This includes certain left-invariant Walker metrics as well as pp-waves.
\end{abstract}
\maketitle 

\newcommand{\gat}{\widetilde{\gamma}}
\newcommand{\Gat}{\widetilde{\Gamma}}
\newcommand{\thet}{\widetilde{\theta}}
\newcommand{\Thet}{\widetilde{T}}
\newcommand{\rt}{\widetilde{r}}
\newcommand{\st}{\sqrt{3}}
\newcommand{\kat}{\widetilde{\kappa}}
\newcommand{\kz}{{K^{{~}^{\hskip-3.1mm\circ}}}}
\newcommand{\gpe}{{g_{_{PE}}}}
\newcommand{\upe}{{\mathcal U}_{_{PE}}}
 

\newcommand{\vect}[1]{\mathbf{#1}}
\renewcommand{\k}{\mathbf{k}}
\renewcommand{\l}{\mathbf{l}}
\renewcommand{\e}{\mathbf{e}}
\renewcommand{\f}{\mathbf{f}}
\newcommand{\bprf}{\begin{proof}}
\newcommand{\eprf}{\end{proof}}


\section{Introduction and  main results}

This paper paper is a follow-up of our papers \cite{leistner-nurowski08,leistner-nurowski09,  AndersonLeistnerNurowski15}, where we presented several examples of pseudo-Riemannian conformal structures, not conformally Einstein, with explicit Ricci-flat 
Fefferman-Graham ambient metrics. The Fefferman-Graham ambient metric  is a fundamental construction from conformal geometry which is defined as follows:

Given  a conformal class represented by a metric $g$ on a smooth $n$-dimensional manifold $M$,
 a {\em Fefferman-Graham ambient metric} or just an {\em ambient metric}  is a metric
\begin{equation}\label{ambient-intro}
	\widetilde{g}=2\, \der t  \der (\rho t)+t^2(g(x^i)+ h(x^i,\rho)), 
\end{equation}
defined on 
$\widetilde{M}=(0,\infty)\times M\times(-\epsilon,\epsilon)$, with coordinates $x^i$ on $M$, $t\in(0,\infty)$ and $ \rho\in (-\epsilon,\epsilon)$,
such that 
$h(x^i,\rho)_{|\rho=0}=0
$ and 
 \begin{equation}\label{fg-eqs}
	Ric(\widetilde g) = O(\rho^m),\quad\text{ with $m=\infty$ if $n$ is odd and $m=\tfrac{n-2}{2}$ if $n$ is even. }
\end{equation}
 Fefferman and Graham \cite{fefferman/graham85,fefferman-graham07} have shown that an ambient metric always exists and is unique in a certain sense, which justifies it to call it  {\em the} ambient metric.  
 Moreover, when $n$ is even, there is a conformally covariant, divergence and trace free $(0,2)$-tensor $\cal O$, the {\em Fefferman-Graham obstruction tensor}, which vanishes whenever (\ref{fg-eqs}) holds also for $m\ge \frac{n}{2}$.
 
 We will refer to the  equations (\ref{fg-eqs}) for a metric of the form \eqref{ambient-intro} as the {\em Fefferman-Graham equations}.  Sometimes we will say that a solution of (\ref{fg-eqs}) is given by $h$, by which we mean that $h$ defines a metric $\tg$ via the formula (\ref{ambient-intro}) such that $Ric(\tg)=O(\rho^m)$. Moreover if equation (\ref{fg-eqs}) holds for all $m$ when $n$ is even, we emphasise this by call $\tg$ a {\em Ricci-flat ambient metric}.

%

Finding  explicit (Ricci-fat) ambient metrics amounts to solving a system of second order  PDEs for the unknown symmetric $\rho$-dependent $(0,2)$-tensor field $h$. In general, these PDE are  nonlinear in $h$, however, 
for the examples presented in \cite{leistner-nurowski08,leistner-nurowski09, AndersonLeistnerNurowski15},
we were able to solve these PDEs explicitly, by the following approach:
We found an ansatz for $h$ such that  the operator  $Ric(\widetilde g)$ became linear in $h$, which allowed us to solve the equation $Ric(\widetilde g)=0$.
This raises the immediate question: what are the features of the conformal class responsible for this phenomenon? In the  present paper we will identify one of these features  as  a property of the conformal holonomy:
\begin{theorem}\label{theo1}
Let $(M,[g])$ be a conformal manifold  such that the conformal holonomy  admits an invariant subspace that is totally null  and of  dimension greater than $1$. Then there is a metric $g$ in the conformal class defining the linear differential operator $\cal A$ acting on $\rho$-dependent symmetric bilinear forms, \begin{equation}\label{linop}
\cal A_{ij}(h)=
2\rho \ddot h_{ij}+(2-n) \dot h_{ij}
+  
2 R^{k\ \ l}_{~ij~}h_{kl}
 - \Box h_{ij},
 \end{equation}
 where  $R_{ijkl}$ is the curvature tensor and $\Box=\nabla^k\nabla_k
$ the tensor Laplacian of $g$ and  the dot denotes the  derivative with respect to $\rho$,
  such that a solution of equation (\ref{fg-eqs})  
  is given via (\ref{ambient-intro}) by a divergence free symmetric bilinear form $h$ that solves the equation
  \begin{equation}
\label{quadeq}
 h^{kl}\nabla_k\nabla_lh_{ij}+
\nabla_{k}h_{li}\nabla^lh_{~j}^{k}
+
\cal A_{ij}(h)+
2R_{ij}
=O(\rho^m),
%
\end{equation}
with   $m=\infty$ if $n$ is odd and $m=\tfrac{n-2}{2}$ if $n$ is even, and where
  $R_{ij}$ is the Ricci tensor of $g$. 
\end{theorem}
Although the
appearance of the quadratic terms in equation (\ref{quadeq}) is somewhat unsatisfactory, in many cases there is an  ansatz for $h$ such that the quadratic terms vanish and the resulting linear equation for $h$ can be solved explicitly. We will come back to this.

Recall that the conformal holonomy is defined as follows. To a conformal class of signature $(p,q)$ one can assign the normal conformal  $\mathfrak{so}(p+1,q+1)$-valued Cartan connection which induces a principal connection and in turn a metric connection on the vector bundle of conformal standard tractors. The conformal holonomy group is the holonomy group of this connection and its natural representation on $\mathbb{R}^{p+1,q+1}$ is the {\em conformal holonomy representation} or simply the {\em conformal holonomy}. In analogy to Riemannian geometry, the reduction of the conformal holonomy to a proper subgroup of $\mathbf{SO}(p+1,q+1)$ is related to the existence of special structure of the conformal manifold, such as the existence of Einstein scales \cite{Gauduchon90,bailey-eastwood-gover94}, the structure of a  Fefferman space \cite{Fefferman76}, twistor spinors \cite{BFGK91, leitnerhabil}, or exceptional conformal structure \cite{nurowski04,bryant06}. One feature that is very different to Riemannian holonomy reductions is that the same conformal holonomy reductions can induce different structures along different {\em curved orbits} \cite{cgh11}. 
There is however a close relationship between the conformal holonomy and the holonomy of the Levi-Civita connection of the Fefferman-Graham ambient metric \cite{CapGoverGrahamHammerl15} and in the present paper we will analyse this relationship further for a specific class of conformal structure that plays an important role for the classification  of conformal holonomies.

 If the conformal holonomy representation is irreducible, several classification results are known \cite{Alt12,Di-ScalaLeistner11,alt-discala-leistner14}. In the case when the holonomy representation is  not irreducible, three essentially different situations have to be distinguished:  the invariant subspace is a) of dimension one, b) of dimension greater than one and  non-degenerate, or c) of dimension greater than one and degenerate with respect to the tractor metric. For case a) it is well known that, locally on an open and dense set in $M$, there is an Einstein metric in the conformal class. In this case there is an explicit Ricci-flat ambient metric,  see Remark \ref{einsteinremark} below.   Case b) is similar, here 
there is a metric in the conformal class that is a product of Einstein metrics (with related Einstein constants), \cite{armstrong07conf,leitner04, ArmstrongLeitner12}.  
Again, such conformal structures admit Ricci-flat ambient metrics \cite{GoverLeitner09}.  
 The last case c), when the invariant subspace is degenerate, can be reduced to the situation in Theorem \ref{theo1}: intersecting the invariant subspace with its orthogonal space gives a holonomy invariant totally null space. 
It was shown in a 
series of papers \cite{leistner05a,leistner-nurowski12,Lischewski15} that
the assumption in Theorem \ref{theo1} --- that the conformal holonomy admits an invariant totally null subspace of dimension  $k+1>1$ --- is equivalent to the the existence,
locally and outside a singular set, of a totally null distribution $\mathcal N$ of rank $k$
and a  metric $g$ in the conformal class, such that:
\begin{enumerate}
\item[(A)]
The image of the Schouten tensor $\P$ of $g$ is contained in $\mathcal N$ (which  implies  $\mathsf{P}^2=0$),
  \item[(B)] $\mathcal N$ is parallel (with respect to the Levi-Civita connection of $g$). 
  \end{enumerate}  
 In the present paper we will deal with the problem of finding ambient metrics for such conformal classes. 
  
Metrics with 
 a parallel totally null distribution $\cal N$ are called {\em Walker metrics} \cite{walker50I}.
Metrics with properties (A) and (B) are special Walker metrics, for which  the image of the Schouten tensor is contained the parallel null distribution $\mathcal N$. This  implies that these metrics are  scalar flat and hence the Schouten tensor is a constant multiple of the Ricci-tensor. In particular, their Ricci and Schouten tensors are divergence free.  In the following we will call 
metrics that have both properties  (A) and (B)  {\em null Ricci Walker metrics}, referring to the property that  image of the Ricci tensor is totally null.
The case $k=1$  was considered in \cite{leistner-nurowski12}, 
where the metrics were called {\em  pure radiation metrics with parallel rays}.
There are many known examples of null Ricci Walker metrics. This includes Lorentzian pp-waves but also the examples of metrics  we gave in \cite{AndersonLeistnerNurowski15},  which are of  signature $(3,3)$ and lie in Bryant's conformal classes \cite{bryant06}. 
Recently, in   \cite{HammerlSagerschnigSilhanTaghavi-ChabertZadnik16} the ambient metric for {\em Patterson-Walker metrics} was computed. Patterson-Walker metrics are null Ricci Walker metrics in neutral signature $(n,n)$ that arise from projective structures in dimension $n$.
In Section \ref{examples} we will give more examples of null Ricci Walker metrics including left-invariant metrics.
 
With  the above characterisation of the assumption,   Theorem \ref{theo1} is a consequence of several results we will prove in this paper. To explain these results, we recall that all the examples in \cite{AndersonLeistnerNurowski15} satisfy property (A). This observation combined with the fact that $\del_\rho h|_{\rho=0}=2\P$, suggested our ansatz for $h$ as a tensor satisfying $\Im(h)\subset \cal N$.
If not only (A) but also (B) is satisfied, which is the case for most but not all of the examples in \cite{AndersonLeistnerNurowski15}, then we can show that the condition $\Im(h)\subset \cal N$ is necessary:
\begin{theorem}\label{theo2intro}
Let $(M,g)$ be a pseudo-Riemannian null Ricci Walker metric with parallel null distribution $\mathcal N$.
Then for every ambient metric 
$\widetilde{g}=2\, \der t  \der (\rho t)+t^2(g(x^i)+ h(x^i,\rho))$, i.e., a solution for the equations (\ref{fg-eqs}), 
it holds
\begin{equation}
\label{theo2introeq}
\mathrm{div}^g (h)=O(\rho^m),\quad \Im(h)\subset \cal N \mod O(\rho^m)
\end{equation}
with $m=\infty$ when $n$ is odd and $m=\frac{n}{2}$ when $n$ is even.
Moreover, when $n$ is even, the obstruction tensor satisfies $\Im(\cal O)\subset \cal N$, and there is an ambient metric for which $h$ satisfies equations (\ref{theo2introeq}) for  $m=\infty$.
\end{theorem}
We will prove this theorem in Section \ref{theo2sec}. Note that the statement about the obstruction tensor  can also be obtained from results in \cite{LeistnerLischewski15}.
Theorem \ref{theo2intro} leads us to study the equation~(\ref{fg-eqs}) for $\widetilde g$ as in (\ref{ambient-intro}) defined by a Walker metric $g$ with parallel null distribution $\cal N$ and with a tensor $h$ with $\Im(h)\subset \cal N$. From the results and computations  in Section~\ref{towards} and Section \ref{walkersec} we obtain the following statement, which together with Theorem \ref{theo2intro} implies Theorem~\ref{theo1}:
\begin{theorem}\label{theo3intro}
Let $(M,g)$ be a null Ricci Walker metric with parallel null distribution $\cal N$ and assume that $h$ is a divergence-free symmetric $(0,2)$-tensor field such that $\Im(h)\subset \cal N$.
Then the metric $\widetilde g$  defined by $h$ via equation (\ref{ambient-intro}) satisfies (\ref{fg-eqs}) if and only if $h$ satisfies equation~(\ref{quadeq}).
\end{theorem}

In the case when the parallel null distribution $\cal N$ has rank one or satisfies an additional condition on the curvature,  we can strengthen this result in the sense that the  quadratic terms in equation (\ref{quadeq})  vanish:

\begin{corollary}\label{introcorollary}
Let $(M,[g])$ be a conformal manifold  
given by a null Ricci Walker metric $g$ with parallel null distribution $\cal N$ that has  rank one, or satisfies $\cal N\hook R=0$, for $R$ the curvature tensor  of $g$.
Then there is an ambient metric, i.e.,   
 a solution of (\ref{fg-eqs}), that is given via (\ref{ambient-intro})  
by a divergence free symmetric bilinear form $h$ that solves the linear system of PDEs
  \begin{equation}
\label{lineq}
\cal A_{ij}(h)+2R_{ij}
=
O(\rho^m),
\end{equation}
%
with $m=\infty$ if $n$ is odd and $m=\tfrac{n-2}{2}$ if $n$ is even, and 
where 
  $R_{ij}$ is the Ricci tensor of $g$. When $n$ is even, the obstruction tensor satisfies 
$\Im(\cal O)\subset {\cal N}$ and   $\cal N\hook \nabla\cal O=0$ and is given by
\[\cal O_{ij}=
c_n\  \Box^mR_{ij},\] where $c_n$ is a nonzero constant and  $\Box^m$ is the $m$-th power of the tensor Laplacian.
\end{corollary}
This corollary follows from the previous results by  the following considerations: if the rank of $\cal N$ is one, then $h$ being divergence free implies that 
\begin{equation}\label{lieh}\cal L_Xh=0,\quad\text{ for all }X\in \cal N.\end{equation}
 where $\cal L_X$ denotes the Lie derivative in direction $X$,
and hence that  $\nabla_X h=0$ for all $X$ in $\cal N$, which in turn yields to the vanishing of the quadratic terms in (\ref{quadeq}). Similarly if the rank of $\cal N$ is larger than one, one can show that the curvature condition $\cal N \hook R=0$ implies condition 
\begin{equation}\label{lieP}
\cal L_X\P=0,
\quad\text{ for all }X\in \cal N,
\end{equation}
and consequently  that $\nabla_X \P=0$ for all $X \in  \mathcal N$. This can then be used to show that $h$ has to satisfy the  condition (\ref{lieh}) and which again implies the vanishing of  the quadratic terms.

%
It turns out that for the linearisation of the Fefferman-Graham equations the conditions~(\ref{lieh}) and~(\ref{lieP}) are crucial. In fact, when (\ref{lieP})  is satisfied,  the ansatz (\ref{lieh}) enable us  to  reduce the Fefferman-Graham equations to linear equations in a much larger class  than the one that satisfies the assumptions of Corollary \ref{introcorollary}.
In 
 Section \ref{towards} we show that 
  for $\widetilde{g}$ as in (\ref{ambient-intro}) the Ricci tensor  $Ric(\widetilde{g})$  becomes at most become quadratic in $h$ if we  assume that  
\begin{enumerate}
\item  the image of $\P$ is contained in a totally null distribution $\cal N$,
\item and  that $\cal N^\perp$ is {\em involutive} (but not necessarily parallel). 
\end{enumerate}
The form of the Fefferman-Graham equations in this more general situation, although being at most quadratic in $h$,  is however more complicated than equations (\ref{lineq}). Nevertheless, we find this more general class  noteworthy: the examples of $\mathbf{G}_2$ conformal classes in \cite{leistner-nurowski09,AndersonLeistnerNurowski15}, for which the  linear Fefferman-Graham equation were reduced to linear PDEs,  are {\em not} null Ricci Walker metrics but rather from this more general class. The reduction was possible because these metrics satisfy the additional property 
(\ref{lieP}),
 which suggested the ansatz~(\ref{lieh}).
 
Based on Corollary \ref{introcorollary}, we are able to construct explicit ambient metrics for several examples of  null Ricci Walker metrics, including left-invariant metrics on Lie groups and generalised pp-waves. Our main results in Section \ref{examples} are the following:

 \begin{theorem}\label{biinv-amb-intro}
Let $\lak$ be a two-step nilpotent Lie algebra, 
 $H$ be a Lie group 
with Lie algebra $\lah$, and  $\phi:\lah\to\mathfrak{der}(\lak)$ a Lie algebra homomorphism into the derivations of $\lak$. Let  $G$ be the
 Lie group corresponding to the Lie algebra $\lag$ that is given as the  semi-direct sum
\[\lag=\lah \ltimes_{\phi}\lak.\]
Moreover, let $g$ be a pseudo-Riemannian left-invariant metric on $G$ such that $\z^\perp=\lak$ and $\lag=\lah^\perp\+\z$, where $\z$ is the centre of $\lak$.
Then the conformal class of $g$ on $G$
admits  a
{Ricci-flat}
 ambient metric
\begin{equation}\label{biinv-ambientmetric}\widetilde{g}=2\der(\rho t)\der t+t^2\Big(g+ 
\frac{2 \rho}{n-2}Ric(g)\Big),
%
%
\end{equation}
where  
$n$ is the dimension of $G$ and $Ric(g)$ is the Ricci tensor of $g$.
%
\end{theorem}
We should point out that the ambient metric in (\ref{biinv-ambientmetric}) is not unique (when $n$ even or when non-analytic ambient metrics are allowed). In fact,  in Theorem \ref{biinv-amb} we find the most general form for Ricci-flat ambient metrics for the left-invariant metrics in Theorem \ref{biinv-amb-intro} and show that the ambiguity is parametrised by $\frac{p(p+1)}{2}$ functions of $n-p$ variables, where $p$ is the dimension of $\lah$.

Finally, amongst other results, in Section \ref{examples} we extend our results in \cite{leistner-nurowski08}:

\begin{theorem}\label{pptheo-intro}
Let 
\[
g=2\der u\der v +H\,\der u^2 +\sum_{i=1}^{n-2}(\der x^i)^2
\]
be a Lorentzian pp-wave metric with
 $H=H(x^1, \ldots, x^{n-2},u)$ a function not depending on $v$. Let $\Delta$ is the flat Laplacian in $n-2$ dimensions.
 Then an ambient metric for $[g]$, i.e. a solution of (\ref{fg-eqs}), is given by
 \begin{equation}\label{pp-intro}
 \widetilde{g}=2\der(\rho t)\der t+t^2 g+
 t^2 \left(\sum_{k=1}^{m} \frac{\Delta^k(H)}{k!\prod_{i=1}^k(2i-n)}\rho^k
 +
\sum_{k=0}^\infty \frac{\Delta^k(f)}{k!\prod_{i=1}^k(2i+n)}\rho^{\frac{n}{2}+k}
 \right)\der u^2,
 \end{equation}
 where \oe, and  $f=f(x^1, \ldots, x^{n-2},u)$ is an arbitrary smooth function. Moreover:
 \begin{enumerate}
\item  When $n$ is odd,
$f\equiv 0$  gives the unique Ricci-flat ambient metric that is analytic in $\rho$. 
\item When $n$ is even, the obstruction tensor $\cal O$ is given by
\[\cal O=c_n\, \Delta^{\frac{n}{2}}(H)\, \der u^2,\]
for some non-zero constant $c_n$.  If $\cal O$ vanishes, the metric (\ref{pp-intro}) is Ricci-flat.
\end{enumerate}
\end{theorem}
In addition to this, we obtain   non-analytic  Ricci-flat ambient metrics with $h \downarrow 0$ if $\rho\to0$ from 
formulas (\ref{pro}) and   (\ref{preq}) in Theorems \ref{pftheo} and \ref{pptheo}, in particular in the case when $n$ is even and the obstruction tensor does not vanish.

We believe that the formulas we provide in this paper are useful to obtain explicit solutions to the Fefferman-Graham equations for other examples than the ones given in Theorems~\ref{biinv-amb-intro} and \ref{pptheo-intro}.

\section{The Fefferman-Graham ambient construction}
\subsection{The Fefferman-Graham ambient metric construction}
A \emph{conformal structure} $(M,[g])$ on an $n=p+q$ dimensional manifold $M$ is an equivalence $[g]$ class of $(p,q)$-signature metrics on $M$, such that two metrics $g$ and $\hat{g}$ are in the same class $[g]$ if and only if there exists a function $\phi$ on $M$, such that 
$\hat{g}={\rm e}^{2\phi}g.$

Let us focus on a given conformal structure $(M,[g])$. In the following definition of an ambient metric 
we will refer to a manifold 
 $\widetilde{M}$ that  is  a product
$$\widetilde{M}=(0,\infty)~\times~ M~\times~(-\epsilon,\epsilon),\quad\quad \epsilon>0,$$ 
with respective coordinates $(t,x^i,\rho)$.

\begin{definition}\label{ambientdef}
An \emph{ambient metric  $\widetilde{g}$ for $(M,[g])$ (that is in normal form with respect to $g$)}  is a  metric on $\widetilde{M}$ given by
\begin{equation}\label{fgmetric}
\widetilde{g}=2\der t\der(\rho t)+t^2g(x^i,\rho),\end{equation}
with a 1-parameter family of symmetric bilinear forms $g(x^i,\rho)$ on $M$, parametrized by $\rho$, such that 
$$g(x^i,\rho)_{|\rho=0}=g(x^i),$$
for some metric $g=g(x^i)$ from the conformal structure $[g]$
and such that
\begin{itemize}
\item 
$Ric(\widetilde{g}) = O(\rho^{\infty})$ if $n$ is odd, and 
\item 
$Ric(\widetilde{g}) = O(\rho^{\frac{n}{2}-1})$ and  $\mathrm{tr}_g \left(\rho^{1-\frac{n}{2}} Ric(\widetilde{g})_{|TM \otimes TM}\right) = 0$ along $\rho=0$, if $n$ is even.
\end{itemize}
\end{definition}
The existence and uniqueness result for ambient metrics in \cite{fefferman/graham85,fefferman-graham07} states that for each choice of $g=g(x^i)$ there is an ambient metric  w.r.t.~$g$. In all dimensions $n\geq 3$,  $g(x^i,\rho)$ has an expansion of the form \[g(x^i,\rho) = \sum_{k \geq 0} g^{(k)}(x^i) \rho^k\] starting with
\[ g(x^i,\rho) = g(x^i) + 2\rho\, \mathsf{P}(x^i) + O(\rho^2), \]
where $\P
=\frac{1}{n-2}(Ric-\frac{Scal}{2(n-1)} g)$ is the Schouten tensor of $g=g(x^i)$.
In odd dimensions the Ricci-flatness condition determines $g^{(k)}$ uniquely for all $k$, whereas in even dimensions only the  $g^{\left(k < \frac{n}{2} \right)}$ and the trace of $g^{\left(\frac{n}{2}\right)}$ are determined uniquely.
The ambient metric construction is conformally invariant in the sense that ambient metrics for different metrics in the conformal class are diffeomorphic to each other (modulo $O(\rho^{\frac{n}{2}})$ when $n$ is even).

For $n$ even a conformally invariant $(0,2)$-tensor on $M$, the {\em ambient obstruction tensor} $\mathcal{O}$, obstructs the existence of smooth solutions to $Ric(\tg) = O(\rho^{\frac{n}{2}})$. For $\widetilde{g}$ in normal form w.r.t.~$g$ as in Definition \ref{ambientdef} it is given by
\begin{align}
\mathcal{O} = c_n \left(\rho^{1-\frac{n}{2}} ({Ric(\widetilde{g})}_{|TM \otimes TM})\right)|_{\rho = 0}, \label{defobstr}
\end{align}
where  $c_n$ is some known nonzero constant \cite{fefferman-graham07}. From this one can deduce that $\mathcal{O}$ is trace- and divergence free.

\begin{remark}\label{einsteinremark}
If $[g]$ contains the \emph{flat} metric $g_0$ than the corresponding ambient metric is $$\widetilde{g}=2\der t\der(\rho t)+t^2g_0.$$ 
Similarly, if $[g]$ contains an \emph{Einstein} metric $g_\Lambda$, $Ric(g_\Lambda)=\Lambda g_\Lambda$, then $$\widetilde{g}=2\der t\der(\rho t)+t^2 (1+\frac{\Lambda\rho}{2(n-1)})^2g_\Lambda$$  
is an ambient metric for $[g_\lambda]$ that is Ricci-flat.
\end{remark}

\subsection{The Fefferman-Graham equations}

Given a conformal structure and having its representative $\go$, the search for a corresponding Fefferman-Graham ambient metric  
$$\widetilde{g}=2\der(\rho t)\der t+t^2g(x,\rho),$$
consists in finding a 1-parameter family $g(x,\rho)$ of metrics on $M$ such that the Ricci tensor of the metric $\widetilde{g}$ satisfies equations (\ref{fg-eqs}). 
In Ref. \cite[Eq. 3.17]{fefferman-graham07} the components of $Ric(\widetilde{g})$ for (\ref{fgmetric}) were written explicitly for the unknown tensor $g=g(x^i,\rho)$. Writing $g$ as $g=g_{ij}\der x^i\der x^j$, with $g_{ij}=g_{ij}(x^k,\rho)$, equation (\ref{fg-eqs}) then read as:
\begin{eqnarray}
\rho \ddot{g}_{ij}-\rho g^{kl}\dot{g}_{ik}\dot{g}_{jl}+\tfrac12\rho g^{kl}\dot{g}_{kl}\dot{g}_{ij}-\left(\tfrac{n}{2}-1\right)\dot{g}_{ij}-\tfrac12g^{kl}\dot{g}_{kl}g_{ij}+R_{ij}&=&O(\rho^m),\label{fge1}\\
\label{fge2} g^{kl}\left(\nabla_k\dot{g}_{il} -\nabla_i\dot{g}_{kl}\right)&=&O(\rho^m),\\
\label{fge3} g^{kl}\ddot{g}_{kl}+\tfrac12 g^{kl}g^{pq}\dot{g}_{pk} \dot{g}_{ql}&=&O(\rho^m),
\end{eqnarray}
for $m=\infty$ when $n$ is odd and $m=\frac{n-2}{2}$ when $n$ is even.
Here for each  $\rho$,  $\nabla$ is the Levi-Civita connection of the metric $g(x^k,\rho)=g_{ij}(x^k,\rho)\der x^i\der x^j$,  $R_{ij}$ is the Ricci tensor of $g(x^i,\rho)$, and the dot denotes partial derivative of $g_{ij}$ with respect to $\rho$. The left-hand sides of these equations are the components of the Ricci-tensor $Ric(\widetilde{g})$ of $\widetilde {g}$.

The first of the Fefferman-Graham equations above is a system of nonlinear 2nd order PDEs for the coefficients $g_{ij}$. It is also obvious that finding the general solution for this system with a given initial condition ${g_{ij}}_{|\rho=0}=\go_{ij}$ is rather hopeless.
One can search for Fefferman-Graham metrics assuming that the metric $g(x,\rho)$ admits a power series expansion with integer powers in $\rho$. 
Fefferman and Graham \cite{fefferman-graham07} gave expressions for the first few terms in the power series expansion in $\rho$ of $g(x,\rho)$ so that $\widetilde{g}$ is Ricci-flat up to the order 3. Up to this order, their expansion reads:
$$
g=\go+2\P \rho+\mu\rho^2+\dots,$$ 
with
$\P$ being the Schouten tensor
 for $\go$,   
and
$$(4-n)\mu_{ij}=B_{ij}+(4-n)\P_i^{~k}\P_{kj}.$$ 
Here $B$ is the Bach tensor of the metric $\go$ defined by
$$B_{ij}=\no^k A_{ijk}-\P^{kl}W_{kijl},$$
with 
$$A_{ijk}=\no_j\P_{ki}-\no_k\P_{ji}$$
the Cotton tensor. 
The symbol $\no$ denotes the Levi-Civita connection for $\go$ and $W^i_{~jkl}$ is the Weyl tensor for $g$.

\subsection{Our approach}\label{approach}
Our approach in this paper will be the following: We will write the unknown family of semi-Riemannian metrics $g(x^i,\rho)$ in the Fefferman-Graham metric as 
\[g(x^i,\rho)=\go(x^i)+ h(x^i,\rho), 
\]
where $\go=\go(x^i)$ is a suitable metic from the conformal class (independent of $\rho$) and $h=h(x^i,\rho)$  is symmetric, $\rho$-dependent symmetric bilinear form on $M$. For our approach we will express the Levi-Civita connection and the Ricci tensor of $g(x^i,\rho)$, which is needed in equations (\ref{fge1}, \ref{fge2}, \ref{fge3}), in terms of the Levi-Civita connection and the Ricci tensor of $\go$. 
For this,  recall the formulas  relating the Levi-Civita connections and the curvatures of  two given metrics $g_{ij}$ and $\go_{ij}$. The difference of both Levi-Civita connections is
 given by a tensor field $C^k_{~ij}$,
\be\label{nab}
\nabla_iX_j-\no_iX_j=C^k_{~ij}X_k,\ee
where $X_k$ is a one-form. For vector fields we have
\[\nabla_iX^j-\no_iX^j=-C^j_{~ik}X^k.\]
Since both connections are torsion-free, it is
$C^k_{~ij}=C^k_{~ji}$, which, together with $\nabla_ig_{jk}=0$, implies
\be
\label{c}
C^k_{~ij}=\tfrac{1}{2}g^{kl}\left( \no_lg_{ij} -\no_i g_{jl}-\no_jg_{il}\right)
\ee
For the curvature tensors, defined by $R_{ijk}^{\ \ \ \, l}v_l=2 \nabla_{[i}\nabla_{j]}v_k$ we obtain
\[R_{ijk}^{\ \ \ \,  l}=\Ro_{ijk}^{\ \ \ \, l}+2  \no_{~[i}C_{~j]k}^l +2 C_{~k[i}^pC_{~j]p}^l,\]
and hence for the Ricci tensor
\be
\label{ric}
R_{ij}=R_{ikj}^{\ \ \ k}=\Ro_{ij}+
\no_{i}C_{~kj}^k - \no_{k}C_{~ij}^k +  C_{~ij}^pC_{~kp}^k - C_{~jk}^pC_{~ip}^k.
\ee
We will use these formulas later on.

\bigskip

In the following we will also consider symmetric $(0,2)$-tensors that are symmetric with respect to the metric $\go$ and moreover $2$-step nilpotent. This will be our assumption on the Schouten tensor $\P$ as well  the ansatz for $h$ in the ambient metric.
About such tensors, we recall the following
 algebraic fact:
\blem\label{alglem}
Let $(M,\go)$ be a semi-Riemannian manifold of dimension $n$ and $h\in End(TM)$ be a selfadjoint, i.e., symmetric with respect to $\go$, endomorphism field such that $h^2=0$. Then there is a totally null vector distribution $\cal N$ such that $\Im(h)\subset \cal N$ and $\Ker(h)=\Im(h)^\perp\subset \Ker(h)$. More precisely, there exist a local co-frame 
$
\Theta^1, \ldots ,\Theta^n,
$
such that
\[
 \go=2\sum_{i=1}^p\epsilon_i \Theta^{2i-1} \Theta^{2i}+\sum_{j=2p+1}^{n}\epsilon_{j}(\Theta^j)^2,
 \]
 where $p$ is the dimension of the image of $h$ and  with $\epsilon_i=\pm 1$, and
\[
 h^\flat:=\go(h.,.)=\sum_{i=1}^p\epsilon_i (\Theta^{2i})^2.
\]
\elem
This lemma follows from a result about the normal form of a linear map $h$ that is selfadjoint with respect to non-degenerate bilinear form $\go$, see for example \cite[Theorem 12.2]{LancasterRodman05siam}. 
For $h^2=0$ this result implies that  at a point in $M$, there is a basis $\e_1,\ldots , 
\e_{n}$, 
such that $h$ and $\go$ are given as
\[
h=\begin{pmatrix}
J&0&0&0\\
0&\ddots&0&0
\\
0&0&J&0
\\
0&0&0&0
\end{pmatrix},
\ \ \ 
\go=\begin{pmatrix}
\epsilon_1G&0&0&0\\
0&\ddots&0&0
\\
0&0&\epsilon_pG&0
\\
0&0&0&\mathbf{1}_{(\epsilon_{2p+1},\ldots , \epsilon_{n})}
\end{pmatrix},
\]
in which $h$ consists of $p=\dim(\mathrm{Im}(h))$ Jordan blocks $J:=\begin{pmatrix}0&1\\0&0\end{pmatrix}$, $G:=
\begin{pmatrix}0&1\\1&0\end{pmatrix}$,  and $\mathbf{1}_{(\epsilon_{2p+1},\ldots , \epsilon_{n})}$ is the diagonal matrix with $\epsilon_{2p+1},\ldots , \epsilon_{n}$ on the diagonal. From this we get that $h^\flat=\go(h.,.)$ is given as
$h^\flat(\e_{2i},\e_{2j})=\epsilon_i\delta_{ij}$
and zero otherwise. 
Hence, in the  dual frame $\Theta^i$ defined by $\Theta^i(\e_j)=\delta^i_{~j}$, $h^\flat $ and $\go$ are given as in the lemma. It
 implies
\begin{corollary}
Let $(M,\go)$ be a semi-Riemannian manifold with Ricci tensor $Ric$ and Schouten tensor $\P$.
Then the following are equivalent:
\begin{enumerate}
\item $Ric^2=0$,
\item $\P^2=0$,
\item $\Im(\P)$ is totally null,
\item $\Im(Ric)$ is totally null.
\end{enumerate}
If any of these conditions is satisfied, then 
 $(M,\go)$ has vanishing scalar curvature.
\end{corollary}

\subsection{Necessary conditions for the ambient  metric of null Ricci Walker metrics}

\label{theo2sec}
In this section we will derive conditions on the $h$ of the ambient metric for a conformal class that contains a null Ricci Walker metric $\go$. 
Recall that we defined a {\em null Ricci Walker-manifold}, as a 
 pseudo-Riemannian manifold $(M,\go)$ that admits  a vector distribution $\cal N\subset TM$ of rank $p>0$ such that $\cal N$
is totally null with respect to $\go$, 
 invariant under parallel transport with respect to the Levi-Civita connection $\no$ of $\go$, and contains the image   of the Schouten tensor $\P$, or equivalently of the Ricci tensor.
 Based on the fact that $\dot h|_{\rho-0}=2\P$, our ansatz for $h$ in the following section will be to assume that  
 the image of  $h$  is also contained in $\cal N$. We will now show that for null Ricci Walker metrics this ansatz is in fact necessary, at least up to the critical order when $n$ is even. The following theorem will imply Theorem~\ref{theo2intro} from the introduction.

\begin{theorem}\label{theo2}
Let $(M,\go)$ be a null Ricci Walker metric of dimension $n>2$ with Schouten tensor $\P$ whose image is contained in a $\no$-parallel totally null distribution ${\cal N}$. 
Let 
$\widetilde{g} = 2\der t\, \der(\rho t) + t^2 g$ with $g=g(x^i,\rho)$ be an ambient metric for $\go$ in the sense of Definition~\ref{ambientdef}.
Then for \[h=g-\go
 = \sum_{m \geq 1} \frac{1}{m!} \gtop{m}\, \rho^m\] with $\gtop{m}=\gtop{m}(x^i)$ it holds the following:
\begin{enumerate}
\item If $n$ is odd, then
\begin{align}
\Im \gtop{m} \subset {\cal N}, \label{sta1} \\
\no_k {\gtop{m}}_{i}^{~k} = 0, \label{sta2}
\end{align}
for all $m\ge 1$.
\item If $n$ is even, then (\ref{sta1}) and  (\ref{sta2}) must hold for $m\le \frac{n}{2}-1$ and the  obstruction tensor satisfies  \[\Im(\cal O)\subset \cal N.\]
Moreover, one can choose an ambient metric such that the corresponding $\gtop{m}$ satisfy  (\ref{sta1}) and (\ref{sta2})  for all $m\ge 1$.
\end{enumerate}
\end{theorem}
\begin{remark}
The statement about the obstruction tensor in the case $n$ even can also be obtained from results in \cite{LeistnerLischewski15}.
\end{remark}

\begin{remark}
Note that \eqref{sta1} is equivalent to $\gtop{m}_{ij} = 0$ unless $i,j \in \{n-p+1,...,n\}.$ Moreover, we use the following convention: $g^{kl}$ refers to the inverse of $g_{kl} = g_{kl}(x^i,\rho)$. However, whenever a raised index appears on a coefficient $\gtop{m} = \gtop{m}(x^i)$, the index is raised w.r.t. $\go$, i.e. ${\gtop{m}}^i_{~j}  := {\go}^{ik} \gtop{m}_{kj}$. 
\end{remark} 
\bprf
The proof is carried out by induction over $m$, where we assume $m\le \frac{n}{2}-1$ when $n$ is even. 
When $n$ is odd, we have that $Ric(\tg)=\O(\rho^\infty)$ and when $n$ is even that $Ric(\tg)=O(\rho^{\frac{n}{2}-1})$.

\bigskip

\textit{Step 1:} For $m=1$, the statement follows from the assumption on $\P$ as well as the contracted version of the second Bianchi identity and $\P_{\ i}^i = 0$. 

Assuming the induction hypothesis that the statement holds for $\gtop{b}$ with $ 1 \leq b \leq m-1$, we show that the statement also holds for $\gtop{m}$. As a preparation, note that as a consequence of the induction hypothesis and parallelity of $\cal N$ we have 
\begin{equation}
\label{22}
{\gtop{u}}^{ki} \gtop{v}_{kj} = 0, \quad 
{\gtop{u}}^{ki} \no_j \gtop{v}_{kl} = 0,
\quad
\text{ for all $1 \leq u,v \leq m-1$.}
\end{equation}
Moreover, for the inverse $g^{ij}$ of $g_{ij}$ the induction hypothesis implies that
\begin{equation}
\label{ginv-m-1}
g^{ij}=\go^{ij}- \sum_{p=1}^{m-1} \frac{1}{p!}\htop{p}^{ij}\rho^p +O(\rho^m).
\end{equation}
Indeed, it is
\begin{eqnarray*}
g_{ik}\Big(\go^{kj}- \sum_{p=1}^{m-1} \frac{1}{p!} \htop{p}^{kj}\rho^p\Big)
&=&
\delta_i^{~j}
+\sum_{p,q=1}^{m-1} \frac{1}{q!p!} \htop{q}_{ik}\htop{p}^{kj}\rho^{p+q}+O(\rho^m),
\end{eqnarray*}
so that the first equation in (\ref{22}) verifies (\ref{ginv-m-1}).
Moreover, equations (\ref{22}) and (\ref{ginv-m-1}) then imply that
\begin{align}
\partial_{\rho}^u (C_{~ij}^k)_{\rho = 0} = -\frac{1}{2}{\go}^{kl} \no_i \gtop{u}_{jl} = -\frac{1}{2} \no_i {\gtop{u}}^k_{j}, \label{fur}
\quad\text{ for $i \in \{1,...,n-p\}$ and  $u\le m-1$, }
\end{align}
where the $C_{~ij}^k$ were defined in Section \ref{approach}.

\bigskip

\textit{Step 2:}
Here we show that the induction hypothesis implies that
\begin{align}
 \partial^{a}_{\rho}R_{ij}|_{\rho=0} = 0,
 \quad\text{ for $a\le m-1$ and $i \in \{1,...,n-p\}$, }
 \label{fg1a}
\end{align}
where $R_{ij} $ is the Ricci tensor of $g_{ij}(\rho)$.
To this end, we rewrite this using \eqref{ric} at $\rho = 0$
\begin{align}
\partial^{a}_{\rho}R_{ij} = \partial^{a}_{\rho} \left(\no_i C^k_{~kj} - \no_k C^k_{~ij} + C^q_{~ij}C^k_{~kq} - C^q_{~jk} C^k_{~iq} \right). \label{expansion}
\end{align}
Everywhere, not only at $\rho = 0$, we have $C^k_{~kj} = -\frac{1}{2} g^{kl} \no_j g_{kl}$. Expanding the $g$-s in terms of the $\gtop{u}$  using   the induction hypothesis as well as \eqref{22} and (\ref{ginv-m-1}) reveals that $C^k_{~kj} = O(\rho^{a+1})$. Thus, the first and third term in \eqref{expansion} vanish at $\rho=0$. The fourth term is treated as follows: 

Expanding $\partial_{\rho}^{m-1}\left(C^q_{~jk} C^k_{~iq} \right)$ at $\rho = 0$ gives  a sum of certain coefficients times summands of the form $(\partial_{\rho}^uC^q_{~jk})(\partial_{\rho}^vC^k_{~iq})$ with $u+v \le m-1$. Assuming 
 $i \in \{1,...,n-p\}$ and applying \eqref{fur} to this
yields
\[
(\partial_{\rho}^uC^q_{~jk})(\partial_{\rho}^vC^k_{~iq})
=
\frac{1}{4}
\go^{pq}\go^{kl}\no_j
{\gtop{u}}_{kp}\no_i \gtop{v}_{ql}
=0,\]
since $u$ and $v$ are $\le m-1$ by the induction hypothesis and the fact that $\cal N$ is parallel.
Thus, the fourth term in \eqref{expansion} vanishes at $\rho = 0$. 

Finally, we show that the second term in \eqref{expansion} vanishes at $\rho = 0$: Assuming 
 $i \in \{1,...,n-p\}$ and using \eqref{fur} again, this term is given as
\begin{align}
\frac{1}{2} \no_k \no_i \gtop{a}^k_{~j}. \label{finalterm}
\end{align}
By the induction hypothesis, we must necessarily have that $k \in \{1,...,p\}$. As $\cal N$ is $\no$-invariant it follows for the curvature of $\go$
\begin{align}
\stackrel{0}{R}_{ikhl} = 0 \text{ for all } i \in \{1,...,n-p \}, k \in \{1,...,p\},
\end{align}
see also Lemma \ref{walkerlemma2}.
This shows that the covariant derivatives in \eqref{finalterm} commute  and one obtains $\no_i$ applied to the divergence of $\gtop{a}$, which vanishes by the induction hypothesis. Thus, \eqref{fg1a} is established.

\bigskip

Now we are going to differentiate the Fefferman-Graham equations (\ref{fge1}, \ref{fge2}, \ref{fge3}) with respect to $\rho$ and use that 
\[\partial^k_{\rho} Ric(\tg)=0,\quad \text{ for all $k$ if $n$ is odd, and for $k\le\frac{n}{2}-2$ if $n$ is even.}\]

\textit{Step 3:}
Applying $\partial_{\rho}^{m-2}$ to the third Fefferman-Graham equation \eqref{fge3}, where $\partial_{\rho}$ always denotes the Lie derivative of a tensor in $\rho$-direction, and then evaluating at $\rho = 0$ yields using \eqref{22} that
\begin{align}
{\go}^{kl} \gtop{m}_{kl} =
0,\quad\text{ for all $m$ if $n$ is odd and for $m\le \frac{n}{2}$ if $n$ is even.} \label{trace}
\end{align}

\bigskip

\textit{Step 4:}
We apply $\partial_{\rho}^{m-1}$, for $m\le \frac{n}{2}-1$ if $n$ is even,  to the second Fefferman-Graham equation \eqref{fge2} and evaluate at $\rho = 0$. Using \eqref{trace} and rewriting $\nabla$ in terms of $\no$ and $C$, the result is
\begin{align}
0 &= {\go}^{kl}\no_k \gtop{m}_{il} + c_{u,v,w} \, {\gtop{u}}^{kl} \left( \partial_{\rho}^v (C^h_{~ki}) \gtop{w}_{hl} + \partial_{\rho}^v (C^h_{~kl}) \gtop{w}_{ih} -\partial_{\rho}^v (C^h_{~ik}) \gtop{w}_{hl} -\partial_{\rho}^v (C^h_{~il}) \gtop{w}_{kh}  \right)_{\rho = 0} \label{divfor}
\end{align}
for certain integer coefficients $c_{u,v,w}$, where $u+v+w = m$ and $1 \leq w \leq m-1$. Using $C_{~ij}^k = C_{~ji}^k$ as well as \eqref{22}, the bracket reduces to
\begin{align}
 {\gtop{u}}^{kl} \left( \partial_{\rho}^v (C^h_{~kl}) \gtop{w}_{ih} \right)_{\rho = 0} - \left(\partial_{\rho}^v (C^h_{~il})\right)_{\rho = 0} {\gtop{w}}^{~l}_{h}. \label{summ} 
\end{align}
In order for the second term in \eqref{summ} to be nonzero, we must necessarily have that $l \in \{1,...,p\}$. In this situation, we can insert \eqref{fur} for the $C$-term and it follows using \eqref{22} immediately that the resulting term vanishes. It remains to analyze the first term in \eqref{summ}. Unwinding the definitions, it is  given by
\begin{align}
{\gtop{u}}^{kl} \partial_{\rho}^v \left( g^{hj} \left(\no_j g_{kl} - \no_k g_{jl} - \no_l g_{kj} \right) \right)_{\rho = 0} \gtop{w}_{ih}. \label{formula}
\end{align}
If the $\rho$-derivative falls on $g^{hj}$, then the resulting contraction with $\gtop{w}_{ih}$ is zero by \eqref{22}. Thus $g^{hj}$ in \eqref{formula} can be replaced by ${\go}^{hj}$. But then \eqref{formula} involves a factor ${\gtop{w}}^j_{~i}$, which can only be nonzero if $j \in \{1,...,p\}$, and \eqref{formula} then reduces to 
\begin{align}
{\gtop{u}}^{kl} \no_j \gtop{v}_{kl}  {\gtop{w}}^j_{i} = 0.
\end{align}
Thus, every term in \eqref{divfor} except for the first one vanishes and we obtain $\no_k {\gtop{m}}^k_{~i} = 0$, which establishes \eqref{sta2}. 

\bigskip

\textit{Step 5:}
In order to prove \eqref{sta1}, we apply $\partial_{\rho}^{m-1}$ to the first Fefferman-Graham equation \eqref{fge1}, assume that $i \in \{1,...,n-p\}$ and evaluate at $\rho = 0$. Using the induction hypothesis and \eqref{22} applied to the first-fifth term in the Fefferman-Graham equation \eqref{fge1}, \eqref{trace} applied to the fifth term, as well as 
(\ref{fg1a}), 
we obtain that at $\rho = 0$ and for 
 $i \in \{1,...,n-p\}$ that
\begin{align}
\left(m-\frac{n}{2} \right) \gtop{m}_{ij} + \partial^{m-1}_{\rho}R_{ij}|_{\rho=0}= \left(m-\frac{n}{2} \right) \gtop{m}_{ij}  = 
\partial_{\rho}^{m-1}(
Ric_{ij}(\tg)|_{\rho=0}
.\label{fg1}
\end{align}
If $n$ is odd,
$\partial_{\rho}^{m-1}(
Ric_{ij}(\tg)|_{\rho=0}=0$ for all $m$ and hence
 equation (\ref{fg1})   shows that $\gtop{m}_{ij}=0$ for $i=1, \ldots n-p$ completing the induction and establishing (\ref{sta1}) for all $m$.

If $n$ is even, 
$\partial_{\rho}^{m-1}(
Ric_{ij}(\tg)|_{\rho=0}=0$ for all $m\le \frac{n}{2}-1$, and hence 
equation (\ref{fg1}) shows that $\gtop{m}_{ij}=0$ for $i=1, \ldots n-p$ for all $m\le \frac{n}{2}-1$. But by taking $m=\frac{n}{2}$, it also gives a formula for the obstruction tensor $\cal O_{ij}
$, in which $c_n$ is a non-zero constant:
\[\cal O_{ij}=c_n \del_\rho^{\frac{n}{2}-1}\widetilde{Ric}_{ij}|_{\rho=0}
= 
c_n \partial^{\frac{n}{2}-1}_{\rho}R_{ij} |_{\rho=0}=0,\]
 if $i\in \{1 ,\ldots n-p\}$ by (\ref{fg1a}). 
This verifies the statement about the obstruction tensor.

\bigskip

Finally, in the case that $n$ is even, the terms $\gtop{m}_{ij} $, for $m\ge\frac{n}{2}$,  in an ambient metric   are not subject to any equation and we can choose them to be divergence free and with image in $\cal N$. 
This completes the proof of the Theorem.
\eprf

\section{Towards linear Fefferman-Graham equations}
\label{towards}
In this and the fo
In this section we will compute the Ricci tensor for metrics of the form  
\[\widetilde{\gg}=2\der(\rho t)\der t+t^2(
\ggo+\h),\] where $\h=\h(\rho)$ is a $\rho$-dependent family of symmetric bilinear forms with $\h|_{\rho=0}=0$ and moreover with the property that
\[\Im(\h)\subset \cal N,\]
for a totally null distribution $\cal N$. 
This is motivated by our aim to find the ambient metrics for metrics with two-step nilpotent Schouten tensor $\P$. As we have seen in Lemma \ref{alglem}, if the Schouten tensor $\P$  is two-step nilpotent, 
its image  is contained in a totally null vector distribution $\cal N$.
On the other hand  from  \cite{fefferman-graham07} we know that   $\dot \h |_{\rho=0}=2\P$, which  
leads to our ansatz 
$\Im(\h(\rho))\subset \cal N$ for all $\rho$. We will then  
 successively   impose further conditions on $\cal N$ and on $\h$ so that the Fefferman-Graham equations become at most quadratic and eventually linear in $\h$.

\subsection{Conventions}
\label{conventions}
 In this  and
 in the following sections we work with specific (co)-frames. Hence we will distinguish between tensors  (written in boldface letters)  $\gg$, $\ggo$,  $\h$ and their components $g_{ij}$, $\go_{ij}$, $h_{ij}$  in a {\em specific} (co)-frame that is adapted to $\cal N$  and later on satisfies additional properties. Some of the statements in the next sections  will only hold for the components $h_{ij}$ of $\h$ in such a basis.

Let $\ggo$ be a semi-Riemannian metric and $\cal N$ be a  vector distribution that is totally null and of rank $p>1$.
 We fix a local frame  
 \be\label{frame}
 \e_1, \ldots ,\e_n\quad\text{such that $\e_1, \ldots \e_p$ span $\cal N$ and $\e_1, \ldots , \e_{n-p}$ span $\cal K=\cal N^\perp$.}
 \ee
 We will use the following index conventions:
\be\label{indices}
\begin{array}{rcl}
i,j,k, \ldots &\in &\{1, \ldots, n\}
\\
a,b,c, \ldots &\in &\{1, \ldots , p\}
\\
A,B,C, \ldots &\in &\{p+1, \ldots ,n-p\}
\\
\aa,\bb,\cc \ldots &\in &\{n-p+1, \ldots, n\}.
\end{array}
\ee
We use the indices $i,j,k,\ldots$ as abstract indices (or with respect to an arbitrary frame), whereas indices $\aa, B, \cc$ will refer to components in a frame $\e_a,\e_B,\e\cc$,  such that 
\be\label{index}
\begin{array}{rcl}
\ggo(\e_\aa,\e_b)\ =\ \go (\e_b,\e_\aa)&=&\go_{\aa b}\ =\ \go_{b \aa }\text{ constant and non degenerate,}
\\
\ggo(\e_A,\e_B)\ =\ \go(\e_B,\e_A)&=&\go_{AB}\ =\ \go_{BA}\text{ constant and non degenerate,}
\\
\ggo(\e_i,\e_j)&=&0\text{ otherwise.}
\end{array}
\ee
In other words, if $\Theta^1, \ldots , \Theta^n$ denote the algebraic duals to the $\e_i$'s, i.e.
\[\Theta^i(\e_j)=\delta^i_{~j}
\]
then the metric is
 \be\label{gnilp}
\ggo\ =\ 
\go_{ij}\Theta^i \Theta^j\ =\ 2
\go_{a \cc} \Theta^a\Theta^\cc+\go_{AB}\Theta^A  \Theta^B.
\ee
Note that the inverse $g^{ij}$ of the matrix $\go_{ij}$ is given by $g^{a\bb}=g^{\bb a}$ and $g^{AB}$ satisfying
\[
\go_{a\bb}g^{\bb c}=\delta_{a}^{~c},\qquad
\go_{\aa b}g^{b \cc}=\delta_{\aa}^{~\cc},
\qquad
\go_{AB}g^{BC}=\delta_{A}^{~C}.
\]
This relates the algebraic duals $\Theta^i$ to the metric duals $\ggo(\e_i,.)$ of $\e_i$ as follows
\[
\Theta^a= \go^{a\cc}\ggo(\e_\cc,.),\qquad
\Theta^\aa= \go^{\aa c}\ggo(\e_c,.),\qquad
\Theta^A=\go^{AB}\ggo(\e_B,.)
\]
Now we consider a  symmetric bilinear form $\h$ (possibly depending on a parameter $\rho$) that satisfies  
\[\mathrm{Im}(\h^\sharp)\subset \cal N,\]
where $\h^\sharp$ denotes the metric dual to $\h$, $\h(X,Y)=\ggo(\h^\sharp(X),Y)$.
This is equivalent to $\h$ being of the form
\be\label{hnilp}
\h:= h_{\aa\cc}\, \Theta^\aa \Theta^\cc = h_{ij}\Theta^i\circ \Theta^j,
\ee
i.e., $h_{ij}=0$ unless $i,j=\aa, \cc$, 
for smooth functions $h_{\aa \cc }=h_{\aa \cc}(\rho,x)$ with $h_{\aa\cc}=h_{\cc\aa}$,
 The corresponding $(1,1)$ tensor $\h^\sharp$ has components
\[h_\aa^{~b}=g^{b\cc}h_{\aa\cc}\]
and all others zero,
i.e.
\[\h^\sharp= h_\aa^{~b}\Theta^\aa\otimes \e_b.
\]
and satisfies
\[(\h^\sharp)^2=0,\ \text{ i.e. } h_\aa^{~k}h_{k}^{~b}=0.\]
It holds that 
\[ \cal K=\cal N^\perp \subset \mathrm{ker(\h^\sharp}).\]
Finally, we obtain the $(2,0)$-tensor defined 
by
$h^{ij}=g^{ik}g^{jl}h_{kl}$, i.e.,
with
\[h^{bd}=\go^{b\aa}g^{d \cc}h_{\aa\cc}\]
and all other components zero.
From now on the components of all the tensor are given in the frame (\ref{frame}) with the index conventions as in (\ref{index}).
We have
\blem \label{lem1}
For $\h$ as in (\ref{hnilp})
denote by $\h^{(r)}=(h^{(r)}_{ij})$ the tensor whose components are given by the $r$-th $\del_\rho$-derivative of the components of $h_{ij}$, i.e.,
 $\h^{(r)}:=\del^r_\rho(h_{ij})\Theta^i\circ \Theta^j$. 
Then 
\be\label{hprime}
\go^{ij}h^{(r)}_{ij}=0\ \text{ and }\
h^{(r)}_{ik}h^{(s) k}_{~\ j} =0\text{ for all }0\le r,s.
\ee
Moreover, if $\no$ is the Levi-Civita connection of $\ggo$, then
\be\label{nabh}\no_kh^{(r)}_{ij}=0,\ \text{ unless $i=\aa$ or $j=\aa$},\ee
as well as 
\be\label{gnabh}
\go^{kl}\no_ih^{(r)}_{kl}=0,
\ee
and
\be
\label{hnabh}
h_{~i}^{(r)\, l}\nabla_kh^{(s)}_{jl}=-h_{~j}^{(s)\, l}\nabla_kh^{(r)}_{il}
\ee
for all $r,s=0, 1,\ldots$. 
\elem
\bprf Equations (\ref{hprime}) follows from the fact that $h_i^{~j}$ squares to zero and is trace free.
Equation
 (\ref{nabh}) follows from
\[
\no_X\h(\e_i,\e_j)
\ =
\ X(\h(\e_i,\e_j))-\h(\no_X\e_i,\e_j)
-\h(\e_i,\no_X\e_j)\\
\ = \ 0
\]
unless $\e_i$ or $\e_j$ is equal to $\e_\aa$.

The last equation
(\ref{hnabh})
follows from (\ref{hprime}), 
\[
0=
\nabla_k\Big(
h_{~i}^{(r)\, l}h^{(s)}_{jl}\Big)
=
h_{~i}^{(r)\, l}\nabla_kh^{(s)}_{jl}+h_{~j}^{(s)\, l}\nabla_kh^{(r)}_{il}.
\]
by the  Leibniz rule.
\eprf

\subsection{The Ricci tensor of a  $2$-step nilpotent pertubation}

In the following, for a semi-Riemannian metric  $\ggo$ we will consider perturbations by a $2$-step nilpotent, symmetric bilinear form $\h$ depending on a parameter $\rho$.
By the results in the previous section we can write this perturbation as
\be\label{gnilp1} \gg=\ggo+\h,\quad\text{where}\ \ \ 
\h= h_{\aa\cc} \Theta^\aa\circ \Theta^\cc \ \ \text{ and }\ \ 
\ggo=\go_{a\bb}\Theta^a\Theta^\bb+\go_{AB}\Theta^A\Theta^B
\ee 
where we use the conventions in 
Section \ref{conventions} and with smooth functions $h_{\aa \cc }=h_{\aa \cc}(\rho,x)$ with $h_{\aa\cc}=h_{\cc\aa}$.
The metric coefficients of $\gg$ are 
 $g_{ij}(\rho,x):=\go_{ij}(x)+h_{ij}(\rho,x)$.
The perturbed metric $\gg$ has the property that the inverse of $\gg$ is linear in the perturbation $\h$, i.e., if $g^{ij}$ are the coefficient of the inverse of $g_{ij}$ then 
\begin{equation}\label{invmetric}g^{ij}=\go^{ij}-h^{ij}.\end{equation}
In the following we will raise the indices with $\go_{ij}$.
First we observe:
\begin{proposition}
\label{fglem}
Let $\ggo$ be a semi-Riemannan metric and $\h$ a $\rho$-dependent, $2$-step nilpotent symmetric bilinear form.
Then for the  metric
  \begin{equation}\label{preamb}\widetilde{\gg}=2\der(\rho t)\der t+t^2(\ggo+\h)\end{equation}
  the possibly non-vanishing components of the Ricci tensor are given by
  \begin{eqnarray}
\go^{kl}\no_k\dot h_{il}&\text{ and }&
\label{fg1nilp}
\rho \ddot h_{ij}-(\frac{n}{2}-1)\dot h_{ij}+R_{ij}.
\end{eqnarray}
Here the dots denote the $\rho$ derivatives of the $h_{ij}$'s and $R_{ij }$ are the components of the Ricci tensor of $\gg=\ggo+\h$. 
\end{proposition}
\bprf
The components of the Ricci tensor of  $\widetilde{\gg} $ are  given by the left-hand sides of the  Fefferman-Graham equations (\ref{fge1}, \ref{fge2}, \ref{fge3}).
 Lemma \ref{lem1} shows that the term in the third Fefferman-Graham equation (\ref{fge3}) is zero.

In order to analyse the term in the second Fefferman-Graham equation (\ref{fge2}), we use formula (\ref{nab})  for  expressing $\nabla$ in terms of $\no$ and the tensor $C^k_{ij}=C^k_{ji}$, i.e.,
\begin{eqnarray}
\nonumber
g^{kl}\left(\nabla_k\dot g_{il} -\nabla_i\dot g_{kl}\right)
\nonumber
&=&(\go^{kl}-h^{kl})(\no_k\dot h_{il}-\no_i\dot h_{kl}+C^p_{~kl}\dot h_{ip}-C^p_{~il}\dot h_{pk})
\\
\label{hc}&=&
\go^{kl}(\no_k\dot h_{il}+C^p_{~kl}\dot h_{ip}-C^p_{~il}\dot h_{pk})
-h^{kl}\no_k\dot h_{il}
-h^{kl}C^p_{~kl}\dot h_{ip}
\end{eqnarray}
because $\h$ is trace free and because of Lemma \ref{lem1}. For $C^k_{~ij}$, 
the formula  (\ref{c}) reduces to 
\be\label{ch}
C^k_{~ij}
=
\frac{1}{2}(\go^{kl}-h^{kl})(\no_lh_{ij}-\no_ih_{jl}-\no_jh_{il})
\ee  
again by Lemma \ref{lem1}. Hence
\[
\dot h_{kp}C^p_{~ij}
=
\frac{1}{2}\dot h_{k}^{~l}(\no_lh_{ij}-\no_ih_{jl}-\no_jh_{il})
\]
Therefore the last term in (\ref{hc}) becomes
\begin{eqnarray*}
2h^{kl}\dot h_{ip}C^p_{~kl}
&=&
h^{kl}\dot h_i^{~p}(\nabla_ph_{kl}-\nabla_kh_{pl}- \nabla_lh_{pk})
\\
&=&
-h^{kl}( h_{kl}
\nabla_p\dot h_i^{~p}
-
 h_{pl}\nabla_k\dot h_i^{~p}- h_{pk}\nabla_k\dot h_k^{~p})\ 
=\ 0
\end{eqnarray*}
because of (\ref{hnabh}) in Lemma \ref{lem1}. Similarly, the remaining term in (\ref{hc}) is
\[
\go^{kl}(C^p_{~kl}\dot h_{ip}-C^p_{~il}\dot h_{pk})
-h^{kl}\no_k\dot h_{il}
\ =\ 
-\dot h_i^{~l}\no_kh_l^{~k}
-h^{kl}\no_k\dot h_{il}+\tfrac{1}{2}
\dot h_{kl}\nabla_ih^{kl}
=0.
\]
This verifies the formula for the terms in the second Fefferman-Graham equation.
The term in the first Fefferman-Graham equation (\ref{fge1}) is seen to be  equal to the second term in (\ref{fg1nilp}) by using Lemma \ref{lem1}.
\eprf
The lemma shows that, apart from the Ricci tensor of $\gg$, the Fefferman-Graham equations contain only  terms that are linear in $\h$. Thus, 
we now determine the Ricci tensor of a metric $\gg=\ggo+\h$ in terms of the Ricci tensor of $\ggo$ and of $\h$ using formula (\ref{ric}) and apply this to a  metric $\widetilde{\gg}=2\der(\rho t)\der t+t^2(\ggo+\h)$.
For this  we note that for a metric as in (\ref{gnilp1}) with inverse (\ref{invmetric}) 
the formula \eqref{ric} for the  Ricci tensor of $\gg$ contains terms up to fourth order in $\h$. 
Hence we observe:
\begin{proposition}\label{ricprop} Let $\ggo$ be a semi-Riemannian metric and $\h$ be a $2$-step nilpotent symmetric bilinear form.
The Ricci tensor $R_{ij}$ of $\gg=\ggo+\h$  is given by
\be\label{fullric}
R_{ij}=\Ro_{ij}+  \no^k\no_{(i}h_{j)k} - \frac{1}{2} \no^k\no_kh_{ij}
\ +Q^{(2)}_{ij}(\h)+Q^{(3)}_{ij}(\h)+Q^{(4)}_{ij}(\h)
\ee
in which we raise the indices with $\go_{ij}$ and where the $Q^{(r)}_{ij}(\h)$ are symmetric tensors that are of order $r=2,3,4$ in $h_{ij}$, and which are given explicitly in \eqref{q2}, \eqref{q3} and \eqref{q4} below. 
\end{proposition}

Now we are going to compute the $Q_{ij}^{(k)}(\h)$'s
by using equation \eqref{ric} for the Ricci tensor of $\gg=\ggo+\h$. 
First we not that 
the formula  \eqref{ch} for  $C^k_{~ij}$  and  Lemma \ref{lem1} implies
\[
C^k_{~ki}=-\tfrac{1}{2}(\go^{kl}-h^{kl})\no_ih_{kl}=0.
\]
Hence \eqref{ric} simplifies to
\begin{equation}\label{ricnilp}
R_{ij}=\Ro_{ij} - \no_{k}C_{~ij}^k  - C_{~jk}^pC_{~ip}^k.
\end{equation}
We start with the terms of fourth order in $\h$:
by \eqref{hnabh} in Lemma \ref{lem1} we get
\be
\begin{array}{rcl}
Q^{(4)}_{ij}(\h)\label{q4}&=&
-\tfrac{1}{4}h^{pq}h^{kl}(\no_qh_{jk}-\no_jh_{kq}-\no_kh_{jq})(\no_lh_{ip}-\no_ih_{lp}-\no_ph_{il})\\
&=& 
-\tfrac{1}{4}h^{pq}h^{kl}(\no_qh_{jk}-\no_kh_{jq})(\no_lh_{ip}-\no_ph_{il})\\
&=& 
\tfrac{1}{4}h^{ab}h^{cd}(\no_ch_{jb}-\no_bh_{jc})(\no_\der h_{ia}-\no_ah_{id})\\
&=&
\tfrac{1}{4}h^{ab}h^{cd}(\h( \e_j, [\e_c,\e_b]))
(\h( \e_i, [\e_d,\e_a])),
\end{array}
\ee
where, for the last equality, we have written the summation in terms of the frame field $\e_i$ and used that $h_{ia}=0$.
Note that $Q^{(4)}_{ij}(\h)=0$  if  $[\e_a,\e_b]\in 
 \mathcal K$.

Now we compute the third order terms and  because of \eqref{nabh} in Lemma \ref{lem1} we obtain
\begin{equation}
\begin{array}{rcl}
Q^{(3)}_{ij}(\h)\label{q3}
&=& 
-\tfrac{1}{2}h^{kl}\go^{pq}\left(
\no_ih_{kp}\no_jh_{lq}
-(\no_ph_{ik} -\no_kh_{ip})(\no_qh_{jl} -\no_lh_{jq})
\right)
\\
&=&
-\tfrac{1}{2}h^{ab}\go^{\cc d}\left(
\no_ih_{a\cc}\no_jh_{bd}
-(\no_\cc h_{ia} -\no_ah_{i\cc })(\no_\der h_{jb} -\no_bh_{jd})
\right)
\\
&&-\tfrac{1}{2}h^{ab}\go^{C D}\left(
\no_ih_{aC}\no_jh_{bD}
-(\no_C h_{ia} -\no_ah_{iC })(\no_Dh_{jb} -\no_bh_{jD})
\right)
\\
&=&
\tfrac{1}{2}h^{ab}\go^{\cc d}
(\no_\cc h_{ia} -\no_ah_{i\cc })(\no_\der h_{jb} -\no_bh_{jd})
\\
&&+
\tfrac{1}{2}h^{ab}\go^{C D}\left(
(\no_C h_{ia} -\no_ah_{iC })(\no_Dh_{jb} -\no_bh_{jD})
\right)
\\
&=&
\tfrac{1}{2}h^{ab}\left( 
\go^{\cc d}
(\no_\cc h_{ia} -\no_ah_{i\cc })
\h( \e_j, [\e_d,\e_b])\right)
\\&&{}
+\tfrac{1}{2}h^{ab}\left( 
\go^{C D}(
(\no_C h_{ia} -\no_ah_{iC })(\h(\e_j,[\e_b,\e_D])
\right).
\end{array}\end{equation}
Clearly, this  vanishes if $\left[ \e_a,\e_b \right] \in  \mathcal K$ and $\left[\e_a,\e_B\right]\in  \mathcal K$, and in particular if $ \mathcal K$ is involutive.

Finally, we turn to the second order terms. They are given as
\be
\begin{array}{rcl}
\label{q2}
Q^{(2)}_{ij}(\h)&=&
 \no_k  h^{kl}
\left( \tfrac{1}{2}\no_lh_{ij}
-
\no_{(i}h_{j)l}\right)
+
 h^{kl}\left( \tfrac{1}{2} \no_k\no_lh_{ij}
- \no_k\no_{(i}h_{j)l}\right)
\\
&&{}
-
\tfrac{1}{4} \no_ih^{kl}\no_jh_{kl}
-
\tfrac{1}{4}\left(
\no^kh_{i}^{~l}-\no^lh_{i}^{~k}\right)
\left(
\no_lh_{jk}-\no_kh_{jl}\right).
\end{array}\ee
First we rewrite the last term as
\[
\tfrac{1}{4}\left(
\no^kh_{i}^{~l}-\no^lh_{i}^{~k}\right)
\left(
\no_lh_{jk}-\no_kh_{jl}\right)
=
\no_{[k}h_{l]i}\no^kh_j^{\ l}
=
\no_{[k}h_{l]j}\no^kh_i^{\ l}.
\]
Next, we analyse the  
 term $ h^{kl}
 \no_k\no_{(i}h_{j)l}$ using the divergence of $\h$, Lemma \ref{lem1}, the curvature and the fact that $\h$ is $2$-step nilpotent:
 \begin{eqnarray*}
 h^{kl}
 \no_k\no_{i}h_{jl}
 &=& 
 - h_{jl}\no_k\no_ih^{kl} -\no_kh_{lj}\no_ih^{kl} -\no_kh^{kl}\no_ih_{jl}
 \\
  &=& 
 - h_{jl}\left(\no_i\no_kh^{kl} +h^{pl} \Ro_{ki\ p}^{\ \ k} + h^{kp} \Ro_{ki\ p}^{\ \  l}\right)
  -\no_kh_{lj}\no_ih^{kl} -\no_kh^{kl}\no_ih_{jl}
  \\
  &=& 
 - h_{jl}\no_i\no_kh^{kl}   - h_{j}^{\ l} h^{kp} \Ro_{kilp}
  -\no_kh_{lj}\no_ih^{kl} -\no_kh^{kl}\no_ih_{jl}
 \end{eqnarray*}
 Hence, we obtain 
\begin{eqnarray*} \begin{array}{rcl}
Q^{(2)}_{ij}(\h)&=&
 \tfrac{1}{2} \no_k  h^{kl}
\no_lh_{ij}
+
h_{l(i}\no_{j)}\no_kh^{kl}
+
\tfrac{1}{2} h^{kl}  \no_k\no_lh_{ij}
- h^{kp} h_{\ (i}^{l}  \Ro_{j)klp}
  +\no_kh_{l(i}\no_{j)}h^{kl}
\\
&&{}
-
\tfrac{1}{4} \no_ih^{kl}\no_jh_{kl}
-\no_{[k}h_{l]i}\no^kh_j^{\ l}
\end{array}\end{eqnarray*}
Therefore, if $\h$ is divergence free, i.e. $\nabla_kh^{kl}=0$, we get formula \eqref{q2a} for 
$Q^{(2)}_{ij}(\h)$.

\begin{proposition}\label{linprop0}
Let $\ggo$ be a semi-Riemannian metric 
and $\h$ be a $2$-step nilpotent symmetric bilinear form such that there is a totally null distribution $\cal N$ with  $\Im(\h)\subset \cal N$ and $\cal K=\cal N^\perp$ involutive.
Then
the Ricci tensor $R_{ij}$ of $\gg=\ggo+\h$ is at most quadratic in $\h$, i.e., the terms $Q^{(3)}_{ij}(\h)$ and $Q^{(4)}_{ij}(\h)$ in \eqref{fullric} vanish.
If we assume in addition that $\h$ is divergence free, then 
   \begin{equation}
\begin{array}{rcl}
\label{q2a}
Q^{(2)}_{ij}(\h)=
\tfrac{1}{2} h^{kl}  \no_k\no_lh_{ij}
- h^{kp} h_{\ (i}^{l}  \Ro_{j)klp}
  +\no_kh_{l(i}\no_{j)}h^{kl}
-
\tfrac{1}{4} \no_ih^{kl}\no_jh_{kl}
-
\no_{[k}h_{l]i}\no^kh_j^{\ l}.
\end{array}\end{equation}
\end{proposition}

We can apply these results to the  metric $\widetilde{\gg}=2\der(\rho t)\der t+t^2\gg$ as defined in \eqref{preamb}: Under the assumption that $\cal K$ is involutive and that $\h$ is divergence free we can apply Proposition~\ref{fglem}. Since $\dot \h$ is divergence free if $\h$ is divergence free, it   implies that $\widetilde{\gg}$ is Ricci-flat if and only if
\begin{eqnarray}
\label{fg1nilp2}
\rho \ddot h_{ij}-(\frac{n}{2}-1)\dot h_{ij}
+  \no^k\no_{(i}h_{j)k} - \frac{1}{2} \no^k\no_kh_{ij}
+
\Ro_{ij}+Q_{ij}^{(2)}(\h)&=&0,
\end{eqnarray}
where 
$Q_{ij}^{(2)}(\h)$ is given as in \eqref{q2a}.
Moreover,  that $\h$ is divergence free also  allows us to simplify 
the term $\no^k\no_{(i}h_{j)k}$. In fact,  if $\no^k h_{ik}=0$  we get
 \begin{equation}\label{nknihkj}
 \no^k\no_{i}h_{jk}\ =\ \Ro^{k\ \ l}_{~ij~}h_{kl}+ \Ro^{k\ \ l}_{~ik~}h_{jl}+\no_i\no^k h_{jk}
\ =\ 
 \Ro^{k\ \ l}_{~ij~}h_{kl}+ \Ro^{\ l}_{i}h_{jl}.
 \end{equation}
 This shows that we can eliminate all $\no_i$ derivatives from this term to obtain
 \begin{corollary}
 \label{lincor}
 Let $\ggo$ be a semi-Riemannian metric 
and $\h$ be a $2$-step nilpotent symmetric bilinear form such that there is an involutive distribution $\cal K$ such that $\Im(\h)\subset \cal N=\cal K^\perp\subset \cal K$.
  Then 
the metric $\widetilde{\gg}=2\der(\rho t)\der t+t^2(\ggo+\h)$ is Ricci-flat if the 
perturbation 
 $\h$ is divergence free and 
\begin{eqnarray}
\label{fg1nilp3}
\rho \ddot h_{ij}-(\frac{n}{2}-1)\dot h_{ij} - \frac{1}{2}\Bo h_{ij}
+  
 \Ro^{k\ \ l}_{~ij~}h_{kl}+ \Ro^{k}_{~(i}h_{j)k}
+
\Ro_{ij}+Q_{ij}^{(2)}(\h)
&=&0,
\end{eqnarray}
where $Q_{ij}^{(2)}(\h)$ is given in \eqref{q2a} and $\Bo h_{ij}=\no^k\no_kh_{ij}$.

 \end{corollary}

Now we are looking for  geometric conditions such that $Q^{(2)}_{ij}(\h)$ simplifies further and perhaps vanishes. 
 In fact we show:

\begin{theorem}\label{linprop}
Let $\ggo$ be a semi-Riemannian metric 
and $\h$ be a divergence free, $2$-step nilpotent symmetric bilinear form. 
{If there is an involutive distribution $\cal K$ with $\Im(\h)\subset\cal N= \cal K^\perp\subset \cal K$  and  
\begin{eqnarray}
\no_ZY&\in& \cal K^\perp,\qquad\text{ for all }Y,Z\in \cal K^\perp
\label{nabab}
\\
\label{nabxb}
\no_XY&\in& \cal K,\qquad\text{ for all }X\in TM,Y\in \cal K^\perp,\end{eqnarray}
then, 
\begin{equation}
\begin{array}{rcl}
\label{q2aa}
Q^{(2)}_{ij}(\h)=
\tfrac{1}{2} h^{kl}  \no_k\no_lh_{ij}
-
\no_{[k}h_{l]i}\no^kh_j^{\ l}.
\end{array}\end{equation}
Moreover, if  in addition 
\begin{equation}
\label{liehm}
\cal L_Y\h=0,\qquad \text{ for all }Y\in \cal K^\perp,
\end{equation}}
then
 $Q^{(2)}_{ij}$ is zero, i.e.,  the Ricci tensor of $\gg=\ggo+\h$ is linear in the perturbation $\h$,
 \be\label{fullric1}
R_{ij}=\Ro_{ij}+  \no^k\no_{(i}h_{j)k} - \frac{1}{2} \no^k\no_kh_{ij}.
\ee
\end{theorem}
\bprf We work in a basis $(\e_a,\e_A,\e_\aa)$ and use the conventions as in Section \ref{conventions}.
First note that assumption \eqref{nabxb} implies that 
terms of the form
$\no_kh_{al}$ or $\no_kh_{AB}$ are zero (where we use our index convention). 
This implies that
in formula \eqref{q2a} for $Q_{ij}^{(2)}(\h)$
the terms $\no_kh_{li}\no_{j}h^{kl}$ and 
$
 \no_ih^{kl}\no_jh_{kl}$ vanish.

 Next we look at the curvature term in formula \eqref{q2a} for $Q_{ij}^{(2)}(\h)$. Again by assumption \eqref{nabxb} we have
\[
\Ro(\e_i,\e_a,\e_b,\e_c)
=
-\go(\no_{\e_a}\e_b,\no_{\e_i}\e_c)+ \go(\no_{\e_i}\e_b,\no_{\e_a}\e_c)\]
which vanishes because of \eqref{nabab} and \eqref{nabxb}.
This proves the first statement.

To prove the second point, assumption \eqref{nabxb} gives
\begin{equation}\label{lie}
\no_{[k}h_{l]i}\no^kh_j^{\ l}=
-\tfrac{1}{2}\go^{\aa b}\go^{\cc d}\no_\der h_{\aa i}\no_b h_{\cc j}
+\tfrac{1}{2}\go^{AB}\go^{CD}( \h([\e_A,\e_C],\e_i)\no_Bh_{Dj}.
\end{equation}
Note that the last term in this formula is zero since $\cal K$ is involutive. 
On the other hand, we observe that for $Y\in \cal K^\perp$ 
\[
\no_Y\h=
\cal L_{Y}\h,\]
 because of  \eqref{nabxb}. This also shows that  in our situation $\cal L_{Y}\h$ is tensorial in $Y\in \cal K^\perp$. 
If we now assume that $\cal L_{Y}\h=0$ for all $Y\in \cal K^\perp$, then 
$\nabla_Y\h=0$ for all $Y\in \cal K^\perp$ and thus the remaining term
in \eqref{lie} vanishes, as well as the term $h^{kl}\no_k\no_lh_{ij}$. Consequently, $Q_{ij}^{(2)}(\h)$ is zero.
\eprf

Theorem \ref{linprop} gives another
\begin{corollary}
Let $\ggo$ be a semi-Riemannian metric 
and $\h$ be a $2$-step nilpotent symmetric bilinear form. If there is a totally null distribution $\cal N$ such that $\Im(\h)\subset \cal N$, $\cal K=\cal N^\perp$ is involutive and 
conditions \eqref{nabab} and \eqref{nabxb} of Theorem \ref{linprop} are satisfied,
then the metric $\widetilde{\gg}=2\der(\rho t)\der t+t^2(\ggo+\h)$ is Ricci-flat if the following system of linear PDEs on $\h=(h_{ij})$ is satisified:
\begin{eqnarray}
\label{divh}
\mathrm{div}(\h)&=&0,
\\
\label{lieha}
\cal L_Y\h&=&0,\ \forall\ Y\in \cal K^\perp,
\\
\label{fg1nilp3a}
\rho \ddot h_{ij}-(\frac{n}{2}-1)\dot h_{ij}
- \frac{1}{2} \Bo h_{ij}
+  
 \Ro^{k\ \ l}_{~ij~}h_{kl}+ \Ro^{k}_{~(i}h_{j)k}
 +
\Ro_{ij}
&=&0.
\end{eqnarray}
\end{corollary}

The examples of conformal structures in \cite{leistner-nurowski09,AndersonLeistnerNurowski15} satisfy the assumptions of Theorem \ref{linprop} and the corollary, which enabled us to use the ansatz to find Ricci-flat ambient metrics.

Note that the assumptions of Theorem \ref{linprop} imply that $\no_X\e_a\in \mathcal K$ but not that $\mathcal K$ or $\mathcal K^\perp=\mathrm{span}(\e_1, \ldots , \e_p)$ are parallel  distributions. Indeed, the terms
\[
2\ggo(\no_i\e_a,\e_A)\ =\ \go( [\e_i,\e_a ], \e_A) +\go( [\e_A,\e_a],\e_i) +\go ( [\e_A, \e_i], \e_a)
\]
might be non-zero for $i=B$ or $i=\cc$.

\section{Ambient metrics for null Ricci Walker metrics}

\label{walkersec}

In this section we apply the results of the previous section to conformal classes given by a null Ricci Walker metric $\gg$ as defined in the introduction. First we review some results about Walker metrics, then focus on the Ricci tensor, and finally draw the conclusions from the previous sections about the ambient metric of null Ricci Walker-manifolds. Note that in Section \ref{walker}
we drop the suffix $0$ on $\ggo$ for brevity, and use it again in Section~\ref{pwalkeramb} when we need to distinguish between $\ggo$ and the $\rho$-dependent family $\gg$.

\subsection{Walker manifolds} 
\label{walker}
A pseudo-Riemannian manifold $(M,\gg)$ is  
a {\em Walker manifold} if   there is a vector distribution $\cal N\subset TM$ of rank $p>0$ that 
is a totally null with respect to $g$ and 
 invariant under parallel transport with respect to the Levi-Civita connection of $\gg$. The most comprehensive study of Walker manifolds can be found in \cite{walkerbook}. In the following we will derive  a description that is useful for our purpose and allows us to construct examples.
 \begin{proposition}\label{walkerprop}
 Let $(M,\gg)$ be a pseudo-Riemannian manifold of dimension $n$. Then the following conditions are equivalent
\begin{enumerate}
\item\label{walkerprop1}
$(M,\gg)$ is a {\em Walker manifold} with parallel null distribution $\cal N$.
\item \label{walkerprop2}
There exists local coordinates $(x^1, \ldots , x^n)$,  so-called {\em Walker coordinates},
such that
\[
\gg
\
=
\
2 \der x^\aa( \delta_{\aa b} \der x^b + F_{\aa B}\der x^B+ H_{\aa \bb}\der x^\bb)+ G_{AB}\der x^A\der x^B,
\]
where the $F_{\aa B}$ and $G_{AB}$ independent of the $x^a$'s.
Here we use the same index conventions as in \eqref{indices} as well as $\delta_{\aa b}=1$ if $\aa=n-p+b$ and zero otherwise. In these coordinates, the parallel null distribution $\cal N$  is given by the span of the $\partial_a=\frac{\partial}{\partial x^a}$'s.
\item\label{walkerprop3}
There is a 
 frame 
$(\e_1, \ldots,  \e_n)$ with dual frame $(\Theta^1,\ldots , \Theta^n)$  such that \[\gg\ =\ 
2g_{a \cc} \Theta^a\circ \Theta^\cc+g_{AB}\Theta^A\circ \Theta^B,\]
 with constants $g_{a \cc}$ and $g_{AB}$ 
and such that 
\begin{equation}\begin{array}{rcl}
\label{kint}
\mathcal K&=&\mathrm{span} (\e_1, \ldots \e_{n-p})\text{ is involutive,}
\\
\left[ \e_a,\e_b\right]& = &
\left[\e_a,\e_B\right]\ =\ 0
\\
\left[\e_a,\e_\cc\right]&\in & \mathcal K^\perp,\quad
\left[\e_B,\e_\cc\right]\ \in \  \mathcal K,
\ \text{ and }\ 
\left[\e_\aa,\e_\cc\right]\ \in \  \mathcal K^\perp
\end{array}
\end{equation}
In this frame $\cal N=\cal K^\perp=\mathrm{span}(\e_1, \ldots ,\e_p)$.
 \end{enumerate}
 
 \end{proposition}
 
 \bprf 
 The equivalence of items \eqref{walkerprop1} and  \eqref{walkerprop2} is due to Walker \cite{walker50I}. In order to show that \eqref{walkerprop2} implies \eqref{walkerprop3}, we fix some 
 Walker coordinates
$(x^1, \ldots , x^n)$ such that
\[
\gg
\
=
\
2\der x^\aa( \delta_{\aa \bb} \der x^b + F_{\aa B}\der x^B+ H_{\aa \bb}\der x^\bb)+ G_{AB}\der x^A\der x^B 
\]
with $F_{\aa B}$ and $G_{AB}$ independent of the $x^a$'s. Then we set
\[
\e_a:=\partial_a, \qquad
\e_A:=
C_{A}^{~B}\left( \del_B- F_{\aa B} \delta^{\aa\bb }\del_b\right),
\qquad
\e_\cc:=
\partial_\cc- H_{\aa\cc}\delta^{\aa\bb}\del_b,
\]
where $C_A^{~B}$ is a matrix such that $C_A^{~B}G_{BE}C_D^{~E}=\delta_{AD}$. Note that, since $G_{AB}$ does not depend on the $x^a$'s, also $C_A^{~B}$ does not depend on the $x^a$'s. We claim that this frame satisfies all the conditions \eqref{kint}. Clearly, the metric in this frame has the right form and $[\e_a,\e_b]=0$. But also the other commutator relations are satisfied:
\begin{eqnarray*}
\left[\e_a,\e_\cc\right]&=&\left[  \del_a, \del_\cc- H_{\bb\cc}\delta^{\bb\bar{e}}\del_e
\right]
\ =\ 
-
\der H_{\bb\cc }( \del_a) \delta^{\bb\bar{e}}\del_e \ \in \  {\mathcal K}^\perp
\\
\left[\e_a,\e_A\right]&=&\left[  \del_a , C_{A}^{~B}\left( \del_B- F_{\cc B} \delta^{\cc\dd }\del_d\right)
\right]
\ =\ 
0
\\
\left[\e_\cc,\e_A\right]&=&\left[  \partial_\cc- H_{\aa\cc}\delta^{\aa\bb}\del_b
,
  C_{A}^{~B}\left( \del_B- F_{\bar{e} B} \delta^{\bar{e}\dd }\del_d\right)
\right]
\
\in \
\mathcal K
\end{eqnarray*}
This shows that all the conditions \eqref{kint}  are satisfied. 

Conversely, we have to show that the bracket relations \eqref{kint} imply that $\nabla_X\e_a\in \cal N=\cal K^\perp$. For this we use the Koszul formula
\[2\gg(\nabla_{\e_i}\e_a,\e_j)= \gg([\e_i,\e_a],\e_j)+\gg([\e_j,\e_a],\e_i)+\gg([\e_j,\e_i],\e_a).\]
From \eqref{kint} is follows that this is zero for all $j=a$ and $j=B$. Hence, $\nabla_X\e_a\in \cal N=\cal K^\perp=\mathrm{span}(\e_1, \ldots, \e_p)$.
 \eprf

Next we record formulas for the curvature of a Walker metric.

\blem\label{walkerlemma2}
Let $(M,\gg)$ be a Walker manifold and let $(\e_1, \ldots ,\e_n)$ be a frame as in \eqref{walkerprop3} of Proposition \ref{walkerprop} such that $\gg$ is given as in \eqref{gnilp}.
\begin{enumerate}
\item
Let $\Gamma_{~~ij}^k$ the connection components with respect to the frame $(\e_1, \ldots , \e_n)$, i.e., defined by
$\nabla_i\e_j=\Gamma_{~~ij}^k\e_k$.
Then 
\begin{equation}\label{christoffelwalker}
\begin{array}{rcl}
\Gamma_{~ab}^k\ =\ \Gamma_{~ba}^k\ = \
\Gamma_{~Ab}^k\ =\ \Gamma_{~bA}^k&=&0,
\\[1mm]
\Gamma_{~ai}^B\ =\ \Gamma_{~ia}^B
&=&0,
\\[1mm]
\Gamma_{~ai}^\cc\ =\ \Gamma_{~ia}^\cc
\ =\ 
\Gamma_{~Ai}^\cc\ =\ \Gamma_{~iA}^\cc
&=&0.
\end{array}
\end{equation}
\item
The curvature tensor and and the Ricci tensor of $\gg$ satisfy
\begin{equation}
\label{curvwalker}
 R_{ijab}\ =\ R_{ijaB}\ =\ 0,
 \end{equation}
 and 
 \begin{equation}
\label{ricwalker0}
R_{a b} = R_{a B} =0,
\end{equation}
for all $a,b=1, \ldots , p$, $B=p+1, \ldots n-p$ and $i,j=1, \ldots n$.
\end{enumerate}
\elem

\bprf
The properties of the connection components are a direct consequence of $\mathcal K$ and $\mathcal K^\perp $ being  parallel distributions and of the Koszul formula
\[
\Gamma_{~~ij}^k=\tfrac{1}{2}g^{kl}\left( \gg( [ \e_i,\e_j],\e_l) + \gg( [ \e_l,\e_j],\e_i) + \gg( [ \e_l,\e_i],\e_j)\right)
.\]
As $\mathcal K$ and $\mathcal K^\perp $ are parallel distributions, in the given frame,  the curvature tensor of  a Walker manifold satisfies equations (\ref{curvwalker}).
Indeed, we have for example
\[
R_{b i A d}=\gg(\Ro(\e_b,\e_i)\e_A,\e_d)=0\]
since $\mathcal K$ is parallel and thus $R(\e_b,\e_i)\e_A\in \mathcal K$.
 This implies that the components of the Ricci tensor
\[
R_{ai}
=
g^{b\cc} (R_{bai\cc}+R_{\cc aib})+g^{AB}R_{AaiB}
=
g^{b\cc}R_{\cc aib}
\]
are zero unless $i=\dd$.
\eprf

This shows that the terms of the Ricci tensor that could prevent a Walker metric from being null Ricci Walker are the following
\begin{equation}
\label{ricwalker}
\begin{array}{rcl}
R_{\aa b}&=&g^{\cc d}R_{d\aa b\cc}
\\[1mm]
R_{AB}&=&
g^{CD}R_{CABD}
\\[1mm]
R_{\aa B}&=&g^{\cc d}R_{\cc B\aa d}+g^{AC}R_{AB\aa C}.
\end{array}
\end{equation}
We will now give conditions for these terms to vanish.
The following results will also
provide a method of constructing examples of null Ricci Walker metrics in Section \ref{examples}, in particular for 
the examples  of Lie groups with left-invariant metric.

\begin{proposition}\label{ricpropwalker}
Let $\gg$ be a metric as in \eqref{gnilp} and assume that the frame $(\e_1, \ldots , \e_n)$ satisfies the following bracket relations
\[
\left[\e_i,\e_j\right]=  r_{ij}^k \e_k
\]
with smooth functions $r_{ij}^k$ satisfying the relations
\begin{equation}
\label{w1} r_{ab}^k \ =\
r_{aB}^k
\ =\ r_{AB}^\cc
\ =\ 
r_{a\cc}^\bb\ =\ r_{a\cc}^B\ =
\
r_{B\cc}^\aa\ =\ 
r_{\aa\cc}^{\bar{b}}
\ =\ 
r_{\aa\cc}^B
\ =\ 
0,
\end{equation}
(these are just the conditions in Proposition \ref{walkerprop}). If we assume in addition that
\beq
\label{w5}
r_{AB}^C&=&0,
\eeq 
and 
\beq
 \label{dr1}
\der r_{b\cc}^d(\e_A)&=&0,\\
\label{dr2}
\der r_{BC}^d(\e_A)\  =\ \der r_{ B \cc}^D(\e_A)&=&0
,\eeq
then $\gg$ is a Walker metric whose curvature satisfies in addition 
\[
R_{ABCi}=R_{\aa b D \cc}=0,\qquad R_{Ai}=0,\]
and
\[
R_{ \aa b \cc d}=g_{f(\aa } dr^{f}_{\cc)d}(\e_b)
.\]
Moreover, $\gg$ is null Ricci Walker, if and only if
\begin{equation}
\label{ricinr}
R_{b\cc}=
\tfrac{1}{2} \left( g_{f\cc}g^{\aa d}\der r^f_{\aa d}(\e_b)+\der r^d_{\cc d}(\e_b)\right)=0
.
\end{equation}
\end{proposition}
\bprf

First we compute  the curvature components $R_{bijd}$. Because of the previous lemma  we only have to compute $R_{b\aa\cc d}$ as all other are zero. 
In terms of the $r_{ij}^k$'s the connection coefficients  $\Gamma_{~~ij}^k$ write as
\be\label{christoffel}
\Gamma^k_{ij}=\frac{1}{2}r^k_{ij}+g^{kl}r_{l(i}^mg_{j)m}
=
\frac{1}{2}r^k_{ij}-g^{kl}g_{m(i}r_{j)l}^m
.
\ee
After imposing the condition on the frame to define a Walker metric, i.e., after imposing equations \eqref{w1}, Lemma \ref{walkerlemma2} 
leaves us  with the only possibly non-vanishing  connection coefficients $\Gamma_{~a\cc}^b$, $\Gamma_{~AB}^b$, $\Gamma_{~AB}^C$, $\Gamma_{~A\cc}^b$, $\Gamma_{~A\cc}^B$ and $\Gamma_{~\aa\cc}^k$. 
Imposing the additional condition \eqref{w5}, $r_{AB}^C=0$, implies 
\[
\Gamma_{~AB}^C=-g^{CD}g_{E(A}r_{B)D}^E=0,\]
This together with $\Gamma_{AB}^\cc=0$, implies that $\nabla_{\e_A}\e_B\in \cal K^\perp$ and hence, with $\cal K^\perp$ being parallel,  that 
\[R_{ABCD}=0,\]
and therefore by (\ref{ricwalker}) that
\[R_{AB}=0.\]

Next, we look at the curvature terms in 
$R_{\cc B}=g^{\aa d}R_{\aa B\cc d}+g^{AC}R_{AB\cc C}$ and compute
\[
R_{\aa B\cc d}
=
-g_{\bb d} \der\Gamma^{\bb}_{~\aa\cc}(\e_B)
=
\tfrac{1}{2}g_{b(\aa}dr_{\cc)d}^b(\e_B).
\]
This vanishes because of condition (\ref{dr1}). Moreover, 
\[
R_{ABD\cc}
=
g_{b\cc}\left(\der\Gamma^b_{~BD}(\e_A)-\der\Gamma^b_{~AD}(\e_B)\right)
=
-
g_{b\cc} \der r^b_{D(A}(\e_{B)}) -\der r^E_{D\cc}(\e_{[A})g_{B]E}
+
g_{ED}\der r^E_{\cc[A}(\e_{B]}),
\]
vanishes because of  condition (\ref{dr2}). Hence we have 
\[ R_{ABD\cc}\ =\ R_{\aa B\cc d}\ =\ 0\] and therefore
$R_{Ai}=0$.
Furthermore, because of $\Gamma_{~ab}^k=0$ and $\left[\e_a,\e_\cc\right]=  r_{a\cc}^b\e_b$ we obtain
\[
R_{b\aa d \cc }
=
\gg(\nabla_b\nabla_\aa\e_d,\e_\cc)
=
\left(
\der\Gamma^f_{~\aa d}(\e_b) + \Gamma^k_{~\aa d}\Gamma_{~bk}^f \right) g_{f\cc}
=
\der\Gamma^f_{~\aa d}(\e_b)g_{f\cc}
=
g_{f(\aa} \der r^f_{\cc)d}(\e_b).
\]
which implies the formula (\ref{ricinr}) for the Ricci components $R_{b\cc}$. The metric is null Ricci Walker if and only if these components vanish. This proves the statement.
\eprf

\begin{remark}
Of course, when constructing examples, the $r_{ij}^k$'s in this proposition cannot be chosen freely as they have to obey Jacobi's identity. However  in some situations, such as $\cal N=\cal N^\perp$, i.e., $n=2p$, or when constructing examples of left-invariant metrics, i.e., when the $r_{ij}^k$'s are constant,  the conditions (\ref{w5}), (\ref{dr1}) and (\ref{dr2}) can be imposed without yielding a contradiction.
\end{remark}
\begin{remark}\label{ricrem}
In view of the examples we will construct in Section \ref{examples}, note that in general the remaining Ricci components do not vanish, even if all the $r_{ij}^k$'s are constant:
\beqn
R _{\aa\cc}
&=&
2g ^{b\dd}R _{b(\aa\cc)\dd}+g ^{BD}R _{B(\aa\cc)D}
\\
&=&
2\left(
\der\Gamma_{~(\aa\cc)}^b(\e_b) - \der\Gamma_{~b(\aa}^b(\e_{\cc)})
+ \Gamma_{~b(\aa}^\der r_{\cc)d}^b
- \Gamma_{~b(\aa}^\der\Gamma_{~\cc)d}^b 
+ \Gamma_{~b\dd}^b \Gamma_{~(\aa\cc)}^\dd
\right.
\\
&& +
\left. \der\Gamma^A_{~(\aa\cc)}(\e_A) -\der\Gamma_{~A(\aa}^A(\e_{\cc)}) +\Gamma_{~B(\aa}^Ar^B_{\cc)A}
 +\Gamma_{~B(\aa}^A\Gamma^B_{\cc)A}
+\Gamma^\dd_{~(\aa\cc)}\Gamma_{~A\dd}^A\right).
\eeqn\end{remark}


\subsection{The Fefferman-Graham equations for null Ricci Walker metrics}
\label{pwalkeramb}
Here we apply our results of Theorems \ref{linprop} and \ref{theo2} to null Ricci Walker metrics.  The following theorem will imply Theorem \ref{theo3intro} and consequently Theorem~\ref{theo1} from the introduction.

\begin{theorem}\label{fgwalkerthm}
Let $(M,\ggo)$ be a null Ricci Walker-manifold with parallel totally null distribution $\cal N$ such that $\Im(\P)\subset \cal N$.
Then an ambient  metric $\widetilde{\gg}=2\der t\der(\rho t)+t^2\gg(\rho)$ for $[\ggo]$ in the sense of Definition \ref{ambientdef}  is given by $\gg=\ggo+\h$, where 
 $\h=\h(\rho)$ is divergence free bilinear form  with $\Im(\h)\subset \cal N$ that satisfies  the 
  the PDE
\begin{equation}
\label{fgwalker}
\rho \ddot h_{ij}-\tfrac{n-2}{2}\dot h_{ij} - \tfrac{1}{2}\Bo h_{ij}
+  
 \Ro_{kijl}h^{kl}
+
\Ro_{ij}+
\tfrac{1}{2} \left(h^{kl}  \no_k\no_lh_{ij}
+
\no_{k}h_{li}\no^lh_{~j}^{k}\right)
=O(\rho^m),
\end{equation}
for $m=\infty$ if $n$ is odd and $m=\frac{n-2}{2}$ when $n$ is even. Here 
 $\h=(h_{ij})$, $\Ro_{ijkl}$ denotes the curvature tensor, $\Ro_{ij}$ the Ricci tensor and $\Bo h_{ij}=\no^k\no_kh_{ij}$, all with respect to $\ggo$.
\end{theorem}
\bprf
Let $\widetilde{\gg}=2\der t\der(\rho t)+t^2\gg(\rho)$  be  an ambient metric 
 for the conformal class of $\ggo$
in the sense of Definition \ref{ambientdef}. Then, from Theorem \ref{theo2} we know that there is a $\h=\gg-\ggo$  that is  divergence free and its image is contained in $\cal N$. Then $\h$ and $\cal K=\cal N^\perp$ satisfy the assumptions of Corollary~\ref{lincor}. Hence, the  term quadratic in $\h$ in the Ricci tensor of $\gg=\ggo+\h$ is given by equation \eqref{q2aa}. Note that, since $\cal K$ is parallel, the second term in \eqref{q2aa} simplifies to 
\[\no_{[k}h_{l]i}\no^kh_j^{\ l}=-\tfrac{1}{2}\no_{k}h_{li}\no^lh_{~j}^{k}.\]
Moreover, since $\Im(\P)\subset \cal N$ and $\Im(\h)\subset \cal N$, in \eqref{fg1nilp3} the product of $\h$ with the Ricci tensor of $\go$ vanishes,
\[
\Ro^{k}_{~i}h_{jk}=\tfrac{1}{n-2}\P^{k}_{~i}h_{jk}=0.\]
This proves the statement.
\eprf

This theorem shows  for a null Ricci Walker metric, that the terms in the Fefferman-Graham equations that are non-linear in $\h$ vanish whenever the components $h_{\bb\dd}$ of $\h$ do not depend on the coordinates $x^a$ in Proposition \ref{walkerprop}  corresponding to the total null plane, i.e., if \[\cal L_{\del_a}\h_{\bb\dd}=\del_a(h_{\bb\dd})=0.\] 
In the following we will present two situations in which this assumption is satisfied.

\subsection{Null Ricci Walker metrics with linear Fefferman-Graham equations}
We have seen that the condition (\ref{lieh}), i.e, that 
\[\cal L_X\h=0\quad\text{ for all $X\in\cal \cal N$}\] is crucial for the Fefferman-Graham equations to linearise. We will now see  special classes of null Ricci Walker metrics for which this is the case. 
it turns out that the relation between property (\ref{lieh})  and the the curvature when applied to $\cal N$ is crucial. First we observe:
\begin{lemma}
Let $\ggo$ be a null Ricci Walker metric with parallel null distribution $\cal N$ and Schouten tensor $\P$. Assume furthermore that 
\begin{equation}
\label{Rcond}
X\hook\Ro=0,\qquad\text{ for all }X\in \cal N,
\end{equation}
where $\Ro$ is the curvature tensor of $\ggo$. Then $\cal L_X\P=0$  for all $X\in \cal N$. 
\end{lemma}
\begin{proof}
For Walker manifold, the differential Bianchi identity ensures that  condition (\ref{Rcond}) also implies that $\cal N\hook\no\Ro=0$. This on the other hand implies that 
 $\no_a\P_{ij}=0$, which  for a null Ricci Walker metrics this is equivalent to $\cal L_{\e_a}\P=0$. 
\end{proof}
Next we prove a result that strengthens Theorem \ref{theo2intro} for this class:
\begin{proposition}\label{liehprop}
Let $\ggo$ be a null Ricci Walker metric with parallel null distribution $\cal N$ and Schouten tensor $\P$ satisfying condition (\ref{Rcond}) for its curvature. 

Then an ambient  metric $\widetilde{\gg}=2\der t\der(\rho t)+t^2\gg(\rho)$ for $[\ggo]$ in the sense of Definition \ref{ambientdef}  is given by $\gg=\ggo+\h$, where 
 $\h=\h(\rho)$ 
 satisfies 
 $\Im(\h)\subset \cal N$,  $\cal L_X\h=0$ for all $X\in \cal N$ and solves the linear PDE
 \begin{equation}
\label{fgwalkerh1}
\rho \ddot h_{ij}-\tfrac{n-2}{2}\dot h_{ij} - \tfrac{1}{2}\Bo h_{ij}
+
\Ro_{ij}=O(\rho^m),
\end{equation}
 for all $m$ if $n$ is odd and  for $m\le \frac{n}{2}-1$ if $n$ is even. When $n$ is even, 
 the obstruction tensor is given by
 \[\cal O_{ij}= c_n\  \Bo^m R_{ij},\]
 where $c_n$ is a non-zero constant depending on $n$ and $\Bo^m$ is the $m$-th power of the tensor Laplacian of $\ggo$.
 In particular, 
 \[
 \Im(\cal O)\subset \cal N, \qquad
 \cal L_X\cal O=0, \quad \text{for all $X\in \cal N$.}\]
\end{proposition}

\begin{proof}
From Theorem \ref{theo2} we know that $\h$ in the ambient metric satisfies (or, if $n$ is even, can be chosen such) that $\Im(\h)\subset \cal N$. The remaining properties of 
$\h= \sum_{m \geq 1} \frac{1}{m!} \gtop{m}\, \rho^m$  are proved in a similar way by induction over $m$ as in  the proof of Theorem \ref{theo2}. But now the computations are simplified, as we can use equations (\ref{fgwalker}) in Theorem \ref{fgwalkerthm}, which are equivalent to the Fefferman-Graham equations: 

Applying $\no_a$ to equation (\ref{fgwalker}), differentiating it $(m-1)$ times with respect to $\rho$, for $m\le \frac{n}{2}-1$ when $n$ is even,  and using the induction hypothesis yields
\[
0
=
(m-\frac{n}{2})\no_a\gtop{m}_{ij}
-\frac{1}{2} \go^{kl}\no_a\no_k\no_l \gtop{m-1}_{ij} 
+
\no_a\Ro_{kijl}\gtop{m-1}^{kl}
=
(m-\frac{n}{2})\no_a\gtop{m}_{ij},
\]
Here we use the Bianchi identity and that  (\ref{Rcond}) allows to commute $\no_a$ with $\no_k$. This equation shows $\no_ah_{ij}=\cal O(\rho^m)$ for all $m$ when $n$ is odd and for $m=\frac{n}{2}$ when $n$ is even. 
Moreover, when $n$ is even, the terms $\gtop{m}$ for $m\ge\frac{n}{2}$ are not determined by the Fefferman-Graham equations. So we can choose them in a way that $\cal L_{\e_a}\gtop{m}_{ij}=\partial_a(h_{ij})=0$ which is equivalent to $\no_ah_{ij}$.
With this and the assumption $\Ro_{aijk}=0$, equation (\ref{fgwalker}) reduces to equation (\ref{fgwalkerh1}). Note also that such a $\h$ is divergence free.

In order to obtain the formula for the obstruction tenser when $n$ is even, we write equation (\ref{fgwalkerh1}) in terms of the 
$\gtop{m}_{ij}$ and obtain
\[
m \gtop{1}_{ij}=\Ro_{ij},\qquad 2(k-m)\gtop{k+1}_{ij}=\Bo \gtop{k}_{ij},\quad\text{ for $k=1, \ldots, m-1$.}\]
This shows that the term of order $\rho$ in (\ref{fgwalkerh1}), which is the obstruction tensor, is equal to $c_n\ \Bo^m R_{ij}$ with a nonzero constant $c_n$.
\end{proof}
Note that for $\h=h_{\aa\cc}\Theta^\aa\Theta^\cc$ with  $\no_ah_{ij}=0$ the term $\Bo h_{ij}$, i.e. the wave operator of $\ggo$ applied to the tensor $\h$   in \eqref{fgwalkerh1}, simplifies to
\[\Lo(h_{\bb\dd})=\go^{AC}\no_A\no_C (h_{\bb\dd}),\]
which is the wave operator for the metric $g_{AC}\Theta^A\Theta^C$ in $n-2p$ dimensions applied to the component functions $h_{\aa\cc}$ of $\h$.
Finally, the vanishing of the curvature terms $\Ro_{aijk}$ 
implies that the system \eqref{fgwalkercomp}, in addition to becoming linear, decouples to $\frac{p+1}{2}$ single equations on the  $\frac{p+1}{2}$  components $h_{\bb\dd}$. These equation only differ in  their inhomogeneity:

\begin{corollary}\label{walkercorol}
Let $(M,\ggo)$ be a null Ricci Walker-manifold with parallel totally null distribution $\cal N$ and $\Im(\P)\subset \cal N$ and such that
 that 
\[
X\hook\Ro=0,\qquad\text{ for all }X\in \cal N,
\]
where $\Ro$ is the curvature tensor of $\ggo$. 
Then, the  an ambient metric metric $\widetilde{\gg}=2\der(\rho t)\der t+t^2(\ggo+\h)$  for $[\ggo]$ is given by 
$\h$ whose components $h_{\bb\dd}$ of $\h$ in a basis as in Proposition \ref{walkerprop} satisfy the following inhomogeneous linear PDE
 \begin{eqnarray}
\label{fg1neutral2}
\Delta_-(h_{\bb\cc})+2
\Ro_{\bb\dd}
&=&O(\rho^m),
\end{eqnarray}
where \oe \ and where
 $\Delta_-$ is the linear second order differential operator defined by
 \begin{eqnarray}
\label{fg1neutral2do}
\Delta_-(f)&=&
2\rho \ddot f+(2-n)\dot f
-\Lo(f)
\end{eqnarray}
 for the function $f=f(x^{p+1}, \ldots ,x^n,\rho)$ and with
 $\Lo(f)=\go^{AC}\no_A\no_C (f)=\go^{AC}\e_A(\e_C(f))$.
\end{corollary}

%
A special case of this situation is  when  the parallel null distribution has rank one, i.e.~$ p=1$ and $\cal N=\mathbb{R}\cdot \e_1$. Here the property $\cal L_{\e_1}\h=0$  is directly by $\h=h(\Theta^n)^2$ being divergence free. Indeed, we have 
\[\mathrm{div}(\h)=
\no_kh^k_{~i}
=
\cal L_{\e_1}\h
=
\del_1(h).\]
Moreover, if the rank of $\cal N$ is one, also  the curvature terms $\Ro_{iklj}$ that occur in equation (\ref{fgwalker}) have to vanish:
\begin{lemma}\label{curvricwalker}
If $\gg$ is a null Ricci Walker metric
and if the null parallel distribution $\cal N$ has rank one, then  $R_{\aa b d\cc}=0$.
\end{lemma}
\begin{proof}
This is an immediate consequence of equations (\ref{ricwalker}):
\[
0
=
R_{a\cc}
=
g^{\bb d}( R_{\bb a\cc d}+ R_{d a\cc \bb})+ g^{AB}R_{A a\cc B}
=
g^{\bb d} R_{\bb a\cc d},
\] because of equation (\ref{curvwalker}).
\end{proof}

Hence,  we obtain: 
\begin{corollary}\label{linecor}
Let $(M,\ggo)$ be a null Ricci Walker manifold with a parallel null line $\cal N=\mathbb{R}\cdot \e_1$, a  
a  frame $\e_1=\del_1,\e_B,\e_n$   with a dual frame $\Theta^1,\Theta^B,\Theta^n$ as in Proposition \ref{walkerprop} and such that 
 $\Im(\P)\subset \cal N$, i.e., $\Ric = f (\Theta^n)^2$, for a function $f$ with $\del_1(f)=0$.
Then an ambient  metric $\widetilde{\gg}=2\der t\der(\rho t)+t^2\gg(\rho)$ for $[\ggo]$ in the sense of Definition \ref{ambientdef}  is given by $\gg=\ggo+\h$, where 
 $\h=h(\rho,x^i)(\Theta^n)^2$ 
that satisfies  $\del_1(h)=0 $ and the following linear PDE
\begin{equation}
\label{fgwalkercomp}
\begin{array}{rcl}
\Delta_-(h)+2
f
&=&O(\rho^m),
\end{array}
\end{equation}
where $\Delta_-$ was defined in (\ref{fg1neutral2do}).
\end{corollary}

This corollary and Proposition \ref{liehprop} imply the  
 statements in Corollary \ref{introcorollary}. Note that for null Ricci Walker metrics we have that $\cal L_X\cal O=\nabla_X\cal O=0$ for all $X\in \cal N$. A construction method for  metrics satisfying the assumptions is provided by Proposition~\ref{ricpropwalker}. Explicit examples will be constructed in the next section.


Finally we show an example for which the condition (\ref{Rcond}) is not satisfied and analyse its Fefferman-Graham equations. It turns out that they are not linear in $\h$.
  
\begin{example}
We consider the following Walker metric in signature $(2,2)$ on $M=\mathbb{R}^4\ni(x^1,x^2,y^1,y^2)$:
\[
\gg= 2\der x^1\der y^1 + 2\der x^2\der y^2 + 2(x^1 \der y^1)^2 +2(x^2 \der y^2)^2 -4 x^1x^2 \der y^1 \der y^2\ =\ 
2\left( \Theta^1 \Theta^{\bar{1}}+\Theta^2 \Theta^{\bar{2}}\right)
,\]
where in our notation above we have a co-frame and its dual frame given by
\[
\begin{array}{rclcrcl}
\Theta^1&=&\der x^1 +(x^1)^2 \der y^1 -  2x^1 x^2 \der y^2,
&&
\e_1&=&\tfrac{\partial}{\partial x^1},
\\[2mm]
\Theta^2&=&\der x^2 +(x^2)^2 \der y^2-  2x^1 x^2 \der y^1,
&&\e_2& =& \tfrac{\partial}{\partial x^2},
\\[2mm]
\Theta^{\bar{1}}&=&\der y^1,&&
\quad \e_{\bar{1}}&=&\tfrac{\partial}{\partial y^1} -( x^1)^2
\tfrac{\partial}{\partial x^1}+2x^1 x^2 \tfrac{\partial}{\partial x^2},
\\[2mm]
 \Theta^{\bar{2}}& =& \der y^2, &&
\e_{\bar{2}}&=&\tfrac{\partial}{\partial y^2} -( x^2)^2 \tfrac{\partial}{\partial x^2}
+2x^1 x^2 \tfrac{\partial}{\partial x^1}.
\end{array}\]
This is a Walker metric with parallel null distribution $\cal K=\cal K^\perp=\mathrm{span} (\e_1,\e_2)$. 
Indeed, we have
\[
\nabla \e_1 = 2 (x^1 \der y^1-x^2\der y^2) \otimes \e_1 - 2x^2 dy^1 \otimes \e_2,
\]
and 
\[
\nabla \e_2 = -2 (x^1 \der y^1-x^2\der y^2) \otimes \e_2 - 2x^1 dy^2 \otimes \e_1.
\]
Then by direct computation or using Proposition \ref{ricpropwalker} we see that the 
 Ricci tensor of $\gg$  is given by
\[Ric = -12\left(( x^1 \Theta^{\bar{1}})^2- 4 x^1x^2\Theta^{\bar{1}}\Theta^{\bar{2}}+ ( x^2 \Theta^{\bar{2}})^2 \right),
\]
and hence $g$ is null Ricci Walker. 
The curvature tensor has the following non-vanishing terms
\[R_{1\bar{1}\bar{1}1}= - R_{1\bar{1}\bar{2}2}= R_{2\bar{2}\bar{2}2}=2.\]
Moreover, the Bach tensor, which in dimension $4$ is the obstruction tensor does not vanish,
\[\cal O=-144\left( (x^1\Theta^{\bar{1}})^2- 4x^1 x^2 \Theta^{\bar{1}} \Theta^{\bar{2}}    + (x^2 \Theta^{\bar{2}})^2\right).\]
Hence, there is no smooth Ricci-flat ambient metric and  we can only find an ambient metric  whose Ricci tensor is of first order in $\rho$.
From Theorem \ref{theo2} we know that the ambient metric is of the form 
$\widetilde{\gg} = 2\der t\, \der(\rho t) + t^2 (\gg+\h)$, where 
$\h=\h(x^1,x^2,y^1,y^2,\rho)$ is of the form
\[\h =A (x^1,x^2,y^1,y^2,\rho) (\Theta^{\bar{1}})^2 - 2 B(x^1,x^2,y^1,y^2,\rho) \Theta^{\bar{1}}\Theta^{\bar{2}}) + C(x^1,x^2,y^1,y^2,\rho)  (\Theta^{\bar{12}})^2,\]
with 
$A_{\rho=0}=B_{\rho=0}=C_{\rho=0}=0$ and 
\[0\ =\ \mathrm{div}(\h)
\ =\
\left(\tfrac{\partial A}{\partial x^1} - \tfrac{\partial B}{\partial x^2}\right)\Theta^{\bar{1}}
+ \left(\tfrac{\partial B}{\partial x^1} - \tfrac{\partial C}{\partial x^2}\right)\Theta^{\bar{2}}.
\]
A direct computation
shows that the Fefferman-Graham equations for this example remain non-linear. For example, the $\bar{1}\bar{1}$-component of  equation~(\ref{fgwalker}) is
\[
\rho \ddot{A} -\dot A - 2 A - 
+\tfrac{1}{2}A \tfrac{\partial^2 A}{(\partial x^1)^2}
-
B \tfrac{\partial^2 A}{\partial x^2\partial x^1}
+\tfrac{1}{2}C    \tfrac{\partial^2 A}{(\partial x^2)^2}
+\tfrac{1}{2}\left( \tfrac{\partial A}{\partial x^1}\right)^2
-
\tfrac{\partial A}{\partial x^2}\tfrac{\partial B}{\partial x^1}
+
\tfrac{1}{2}\left( \tfrac{\partial B}{\partial x^1}\right)^2
\]
\[+
(x^1)^2\tfrac{\partial^2 A}{(\partial x^1)^2}
-
4 x^2 x^1 \tfrac{\partial^2 A}{\partial x^2\partial x^1}
-
 \tfrac{\partial^2 A}{\partial y^1\partial x^1}
+
4 x^2
 \tfrac{\partial A}{\partial x^2}
-4 x^1
 \tfrac{\partial A}{\partial x^1}
 +
 (x^2)^2
 \tfrac{\partial^2 A}{(\partial x^2)^2}
-
\tfrac{\partial^2 A}{\partial y^2\partial x^2}
-4x^1
\tfrac{\partial B}{\partial x^2}
 -12(x^1)^2
 \]
In this example
our ansatz (\ref{lieh}), i.e., that
\begin{equation}\label{liehex}\cal L_{\e_1} \h=\cal L_{\e_2}\h=0,\end{equation}
  does not yield a solution to the Fefferman-Graham equations, i.e., to $Ric(\tilde{g})= O(\rho)$. Indeed, the ansatz (\ref{liehex}) is equivalent to the components  $A$, $B$ and $C$ being independent of $x^1$ and $x^2$, and hence  the Ricci tensor of $\tilde{g}$ has the components
\begin{align*}
&\left( \rho \ddot{A}(y^1,y^2,\rho) -\dot A(y^1,y^2,\rho) - 2 A(y^1,y^2,\rho)  -12(x^1)^2\right)(\Theta^{\bar{1}})^2
\\
+& 2\left( \rho \ddot{B}(y^1,y^2,\rho) -\dot B(y^1,y^2,\rho) - 2 B(y^1,y^2,\rho)  -12x^1x^2\right)\Theta^{\bar{1}} \Theta^{\bar{2}}
\\
+&\left( \rho \ddot{C}(y^1,y^2,\rho) -\dot C(y^1,y^2,\rho) - 2 C(y^1,y^2,\rho)  -12(x^2)^2\right)(\Theta^{\bar{2}})^2,
\end{align*}
which cannot be of the form $\rho Q$ for $Q$ a tensor on $M$.
Instead, a solution is for example given by
\[\h =-12 \rho\left((x^1 \Theta^{\bar{1}})^2 -4x^1x^2\Theta^{\bar{1}}\Theta^{\bar{2}} + (x^2 \Theta^{\bar{2}})^2\right),\]
which is divergence free but does not satisfy the ansatz (\ref{lieh}). 
With this $\h$ the  ambient metric $\widetilde{\gg}=2\der(\rho t)\der t+t^2(\ggo+\h)$ has Ricci tensor
\[Ric(\widetilde{\gg})=
-144\rho( 3\rho-1)  \left((x^1\Theta^{\bar{1}})^2
-4 x^1 x^2 \Theta^{\bar{1}}\Theta^{\bar{2}}) +(x^2\Theta^{\bar{2}})^2
\right)
=
\rho( 3\rho-1) \cal O
.\]

\end{example}

\section{Examples with explicit ambient metrics}
\label{examples}
In this section we will provide examples of conformal classes of null Ricci Walker metrics  for which we find  explicit solutions to equation  \eqref{fg1neutral2} obtaining explicit examples of Ricci-flat ambient metrics.
\subsection{Solving the homogeneous equation}

Equation \eqref{fg1neutral2} is a linear, inhomogeneous PDE for each of the functions $h_{\aa\cc}$ given by the linear differential operator
\[\Delta_-=
2\rho \partial^2_\rho +(2-n)\partial_\rho
-\Lo.
\]
In the section we will find metrics for which we get an explicit solution of \eqref{fg1neutral2}. Before this, we start by providing the solution to the homogeneous equation.
\begin{lemma}\label{sol-lemma}
Let $M$ be a smooth manifold of dimension $n$ and $\cal D$ some linear differential operator on $M$. For a function $F\in C^\infty (M)$ we define the functions $F_\pm\in C^\infty(M\times (-\epsilon,\epsilon))$ as
\[F_\pm 
:=\sum_{k=1}^{\infty} \frac{{\cal D}^k (F)}{k!\prod_{i=1}^k(2i\pm n)}\rho^k,
\]
where $F_-$ is only defined when $n$ is odd or $\cal D^{\frac{n}{2}}(F)=0$.
Moreover, define the following linear differential operators on $ C^\infty(M\times (-\epsilon,\epsilon))$
\[\mathcal D_{\pm}:= 2\rho\del^2_\rho +(2\pm n)\del_\rho -\cal D.\]
Then, for any $F\in C^\infty (M)$ and $f\in  C^\infty(M\times (-\epsilon,\epsilon))$ we have
\beq
\label{d+-}
\mathcal D_\pm(F_\pm) &=& \cal D (F)
\\
\label{d-}
\mathcal D_-(\rho^{\frac{n}{2}} f) & = &\rho^{\frac{n}{2}} \mathcal D_+(f).
\\
\label{d-+}
\mathcal D_-(\rho^{\frac{n}{2}} F_+) \ =\ \rho^{\frac{n}{2}} \mathcal D_+(F_+)&=& \rho^{\frac{n}{2}}  \cal D(F).
\\
\label{solhom}
\cal D_-(\rho^{\frac{n}{2}}(F+F_+))&=&0.
\eeq
In particular, for each $F\in C^\infty(M) $, the function $f=\rho^{\frac{n}{2}}(F+F_+)$ is a solution to the homogeneous equation $\cal D_- (f)=0$.
\end{lemma}
\bprf
To verify equations \eqref{d+-} and \eqref{d-} is a straightforward computation. Both together imply \eqref{d-+} which yields \eqref{solhom}.
\eprf

\subsection{Extensions of nilpotent Lie algebras} Let $\lak$ be a two-step nilpotent Lie algebra of dimension $q$ and let $\laz$ be its centre of dimension $p<q$.
We fix a complement $\lam$ of $\z$,
\[\lak=\z\+\lam\] Then $[\lam,\lam]\subset \laz$ and we can fix a basis $(\e_a)_{a=1, \ldots ,p}$ of $\laz$ and $(\e_A)_{A=p+1, \ldots ,q}$ of $\lam$ such that 
\[
\left[\e_a,\e_b\right] =0,\qquad
\left[\e_a,\e_B\right] =0,\qquad
\left[\e_A,\e_B\right]=r^c_{AB}\e_c,
\]
where $r_{AB}^c$ denote the structure constants of $\lak$. Note that there are no further conditions on these numbers other than $r_{AB}^c=-r_{BA}^c$. 
Denote by $\mathfrak{der}(\lak)$ the derivations of $\lak$  which comes with a canonical Lie algebra structure induced from $\gla (\lak)$. Note that  derivations leave the centre invariant.

Furthermore, let $H$ be a Lie group with Lie algebra $\lah$ and of  dimension $p=\dim (\laz)$ and $\phi:\lah\to \mathfrak{der}(\lak)$ a Lie algebra homomorphism from $\lah$ to the derivations of $\lak$. By fixing a basis $(\e_{\aa})_{\aa=q+1, \ldots , p+q}$ of $\lah$, we can write $\phi$ as
\[
\phi(\e_\aa)\e_b=r_{b\aa}^d\e_d,\qquad
\phi(\e_\aa)\e_B=r_{B\aa}^d\e_d+ r_{B\aa}^E\e_E,
\]
with some constants $r_{b\aa}^d$,  $r_{B\aa}^d$ and $ r_{B\aa}^E$. Finally, with respect to this basis denote the structure constants of $\lah$ by $r_{\aa\bb}^\cc$, i.e., 
\[\left[\e_\aa,\e_\bb\right]\ =\ r_{\aa\bb}^\cc\e_\cc.\]
Now we define the Lie algebra $\lag$ to be semi-direct sum 
$\lag=\lah \ltimes_{\phi}\lak$
of $\lah$ and $\lak$ with respect to $\phi$ of dimension $n=p+q$. Clearly, the structure constants of $\lag$ are given by the numbers
\[r_{AB}^c, r_{b\aa}^d, r_{B\aa}^d, r_{B\aa}^E, r_{\aa\bb}^\cc,\]
which are subject to the conditions $r_{ij}^k=-r_{ji}^k$ and 
\[
r_{AB}^er_{e\cc}^d=
-2r_{\cc[A}^Cr_{B]C}^d,
\]
i.e., that $\phi(\e_\cc)$ is a derivation, as well as
\[
r_{\aa\bb}^\cc r_{d\cc}^e=2r_{d[\aa}^c r_{\bb]c}^e,\qquad
r_{\aa\bb}^\cc r_{A\cc}^d=2r_{A [\aa}^c r_{\bb]c}^d + 2r_{A [\aa}^B r_{\bb]B}^d ,\qquad
r_{\aa\bb}^\cc r_{A\cc}^B=2r_{A [\aa}^C r_{\bb]C}^B,
\]
which ensure that $\phi:\lak\to\mathfrak{der}(\lah)$ is a Lie algebra homomorphism.
The frame $\e_1, \ldots , \e_n$ on the Lie group $G$ corresponding to $\lag$ satisfies the bracket relations of Proposition \ref{ricpropwalker} with the parallel distribution $\mathcal K$ given by $\lak$. Now we define a left invariant metric by formula~\eqref{gnilp}
\[
\gg=g_{a \cc} (\Theta^a\otimes\Theta^\cc+\Theta^\cc\otimes \Theta^a)+\go_{AB}\Theta^A\circ \Theta^B
\]
where the $\Theta^i$'s are again the algebraic duals of the $\e_i$'s and the $g_{ij}$ are constants. Now the distribution $\cal K^\perp$ is given by $\laz$. 
Then Proposition \ref{ricpropwalker} implies that $(G,\gg)$ is a null Ricci Walker manifold of dimension $n$, which, in general is not Ricci-flat. Its possibly non vanishing components are given by  constants $R_{\aa\cc}$.

In order to determine the ambient metric for the conformal class given by $\gg$ on $G$, we have to solve equations \eqref{fg1neutral2} in this setting, i.e., find a functions $h\in C^\infty((-\varepsilon,\varepsilon)\times G)$, such that
\be
\label{fg1neutral2const}
2\rho \ddot h+(2-n)\dot h
-\Delta(h)
+C=0,\ \text{ with initial condition }h|_{\rho=0}\equiv 0,
\ee
with $\Delta(h)=g^{AB}\nabla_A\nabla_Bh$, and for constants $C$ that are given by the components of the Ricci tensor $R_{\aa\cc}$.
Equation \eqref{fg1neutral2const}, when taken along $\rho=0$ implies
\[
\dot h|_{\rho=0}\equiv\frac{C}{n-2}.
\]
Clearly, the problem \eqref{fg1neutral2const} has a linear solution
\[
h=\frac{C}{n-2}\rho,
\]
but Lemma \ref{sol-lemma} shows that there are more solutions. 
From Corollary \ref{walkercorol} we obtain Theorem~\ref{biinv-amb-intro} from the introduction. More precisely, we get
\begin{theorem}\label{biinv-amb}
Let $\lak$ be a two-step nilpotent Lie algebra of dimension $q$ with centre $\laz$ of dimension $p\le q$, and let $H$ be a Lie group of   dimension $p$ and with Lie algebra $\lah$. Let $\phi:\lah\to\mathfrak{der}(\lak)$ a Lie algebra homomorphism into the derivations of $\lak$ and $G$ be the $n=q+p$-dimensional Lie group corresponding to the Lie algebra $\lag$ that is given as the  semi-direct sum
\[\lag=\lah \ltimes_{\phi}\lak,\]
of $\lah$ and $\lak$ by $\phi$.
Fix  a basis $(\e_\aa)_{\aa=1, \ldots ,p}$  of $\lah$, a basis $(\e_a)_{a=1, \ldots ,p}$ of $\laz$ and complement it with  $(\e_A)_{A=1, \ldots ,q-p}$ to a basis of  $\lak$. Let $(\Theta^i)_{i=1, \ldots , n}$ is the dual basis to $(\e_i)_{i=1, \ldots , n}$ and 
\[
\gg=2 \,  g _{a \cc}\, \Theta^a\circ\Theta^\cc+ g _{AB}\Theta^A\circ \Theta^B
\]
be the left-invariant pseudo-Riemannian metric $\gg$ on $G$ defined by real numbers
$
 g _{a\cc}$ and $
 g _{AB}$.
Then the conformal class of $\gg$ on $G$
admits  
{Ricci-flat}
 ambient metrics  given by
\[\widetilde{\gg}=2\der(\rho t)\der t+t^2\Big(\gg+ \Big(
\frac{2 \rho}{n-2} R _{\aa\cc}
+\rho^{\frac{n}{2}}\Big( F_{\aa\cc} + \sum_{k=1}^{\infty} \frac{{\Lo}^k (F_{\aa\cc})}{k!\prod_{i=1}^k(2i+ n)}\rho^k\Big)\Big)
\Theta^\aa\Theta^\cc\Big),\]
where  $ R _{\aa\cc}=Ric^\gg(\e_\aa,\e_\cc)$ are the  components of the Ricci tensor of $\gg$ and $F_{\aa\cc}=F_{\cc\aa}$ are  functions on $G$ with $d F_{\aa\cc}(\e_a)=0$. In particular, when $n$ is odd,
 $F_{\aa\cc}\equiv 0$ gives the unique analytic Ricci-flat Fefferman-Graham ambient metric.
\end{theorem}
Note that in general the metrics $\gg$ as in the theorem are neither Ricci-flat nor do they admit parallel null vector fields  (see also Remark \ref{ricrem}).

\subsection{Generalised pp-waves}
Another class of examples to which our Corollary~\ref{walkercorol} applies  are the Lorentzian pp-waves for which we have determined the analytic ambient metric in \cite{leistner-nurowski08}. The acronym ``pp'' stands for {\em plane fronted with parallel rays}. A Lorentzian pp-wave metric in dimension $n$ is locally given by
\[
\gg
=
2\der u\der v+ H\der u^2 +\sum_{i=1}^{n-2}(\der x^i)^2,
\]
where $H=H(x^i,u)$ is a function that does not depend on $v$.

Here we  generalise this class and the results in \cite{AndersonLeistnerNurowski15,leistner-nurowski08} to higher signature and, more importantly, determine all solutions to the Fefferman-Graham equations including the non-analytic ones, and determine the obstruction tensor in the case of Lorentzian pp-waves.

We will use the same index conventions as in the previous sections ($a=1, \ldots p$, $B=p+1, \ldots n-p$, $\cc=n-p+1,\ldots ,  n$), and define a modified Kronecker delta as
\[ \delta_{\aa b}=\left\{\begin{array}{ll}1,&\text{ if $\aa=b+n-p$}\\ 0&\text{ otherwise.}\end{array}\right.\]
\begin{definition}
Let $\cal U\subset \mathbb{R}^n\ni(x^1, \ldots , x^n)$ be an open set, and $H_{\aa\cc}$ and $G_{AB}$ smooth functions on $\cal U$ satisfying $\det(G_{AB})\not=0$ and 
$\del_a(H_{\aa\cc})=0$ and $\del_a(G_{AB})=\del_\cc(G_{AB})=0$. Then the pseudo-Riemannian metric
\be\label{pp}
\gg
\
=
\
2\delta_{\aa b}  \der x^\aa \der x^b +  H_{\aa \bb}\der x^\aa \der x^\bb+ G_{AB}\der x^A\der x^B 
,\ee
is called a {\em generalised  pp-wave}, or for short, a {\em gpp-wave}. 

If all the $G_{AB}$'s are  constants, we call $\gg$ {\em plane fronted wave with parallel rays}, or for short, {\em pp-wave}.
\end{definition}
To obtain Lorentzian gpp-waves, one sets $p=1$ and $G_{AB}$  positive definite.
For all $p$,  gpp-waves admit $p$ parallel vector fields $\del_a$ and hence are Walker metrics, however in general not null Ricci Walker metrics. As in Proposition \ref{walkerprop}, for gpp-waves 
we have the frame and dual co-frame
\begin{eqnarray*}
\e_a:=\partial_a,
&
\e_B:=
E_{B}^{~A} \del_A,
&
\e_\cc:=
\partial_\cc- H_{\aa\cc}\delta^{\aa\bb}\del_b,
\\
 \Theta^a=\der x^a +H_{\aa\cc}\der x^
 \cc,
 &
 \Theta^B=F^{B}_{~A}\der x^A,
 &
 \Theta^\cc=\der x^\cc,
\end{eqnarray*}
where $E_A^{~B}$ is a matrix such that $E_A^{~B}G_{BC}E_D^{~C}=\delta_{AD}$
and  $F^B_{~A}$ is the inverse of $E_A^{~B}$.
Note that, since $G_{AB}$ does not depend on the $x^a$'s or the $x^\cc$'s neither does $E_A^{~B}$.
The gpp-wave metric  in this frame is 
\[\gg= \delta_{a \bb}\Theta^a\Theta^{\bb} +g_{AB}\Theta^A\Theta^B.
\] 
with $g_{AB}=\epsilon_A\delta_{AB}$. 
 The only non vanishing brackets for this frame are
 \beqn
 \left[ \e_A,\e_B \right]&=&
 -E_{[A}^{~C}E_{B]}^{~D}dF_{D}^{~E}(\del_C) \e_E
 ,
 \\ 
 \left[ \e_A,\e_\bb \right]&=&-d H_{\bb\cc}(\e_A)\delta^{\cc\dd}\e_d
 \\
 \left[ \e_\aa,\e_\bb \right]&=&2\der H_{\cc[\aa}(\del_{\bb]})\delta^{\cc\dd}\e_d,
 \eeqn
Hence, the assumptions of Proposition \ref{ricpropwalker} are satisfied whenever the $G_{AB}$'s are constant, i.e., whenever $\gg$ is a pp-wave.
 
The 
Levi-Civita connection $\nabla$ of a gpp-wave $\gg$ is given by
\begin{eqnarray*}
\nabla_A\e_B&=&\nabla^\mathbf{G}_A\e_B,\ \\
\nabla_\aa \e_B&=& d H_{\aa\cc}(\e_B) \delta^{\cc\bb}\e_b,
\\
\nabla_{\aa}\e_\bb &=&
-2\der H_{\aa[\bb}(\del_{\cc]})\delta^{\cc\dd} \e_d 
-\mathrm{grad}^\mathbf{G}(H_{\aa\bb}),
\end{eqnarray*}
in which $\nabla^\mathbf{G}$ is the Levi-Civita connection of the metric $\mathbf{G}=G_{AB}\der x^A\der x^B$ and $\mathrm{grad}^\mathbf{G}$ the corresponding gradient.
This allows us to compute the curvature, which satisfies $R_{aijk}=0$,  and the Ricci-curvature, whose only possibly non-vanishing terms are given as
\begin{eqnarray*}
R_{AB}&=&R^\mathbf{G}_{AB}
\\
R_{\aa\cc}&=&-\tfrac{1}{2}g^{BD}\gg(\nabla_B(\mathrm{grad}(H_{\aa\cc}),\e_D)\ =\ 
-\tfrac{1}{2}g^{BD}\nabla^\mathbf{G}_B\nabla^\mathbf{G}_D(H_{\aa\cc})=-\tfrac{1}{2}\Delta_\mathbf{G}(H_{\aa\cc}).
\end{eqnarray*}
\begin{lemma}
The defined  gpp-waves satisfy $R_{aijk}=0$ and they are null Ricci Walker metrics if the metric $\mathbf{G}$ is Ricci-flat. 
In particular, pp-waves are null Ricci Walker metrics.
\end{lemma}

\begin{remark}
If we drop the assumption on a pp-wave that the $\e_a$'s are parallel, i.e., that $\partial_a H\not=0$, then the Ricci tensor is no longer two-step nilpotent. For example in the Lorentzian case, i.e., when $p=1$ and $\epsilon_i=1$,  if $\partial_1 H\not=0$ we get that
\[
Ric\left(\partial_1, \partial_n\right)=\partial_1^2(H),\qquad
Ric\left(\partial_A, \partial_n\right)=\partial_A\partial_1(H),
\]
which shows that $Ric$ cannot be two-step nilpotent (see also \cite{leistner05c}).
\end{remark}
 \begin{remark}
Using the necessary conditions that were derived in \cite{gover-nurowski04} for conformal Einstein metrics, a straightforward computation of the Weyl, Cotton and Bach tensors as in \cite{leistner-nurowski08} shows
 that in general gpp-waves are not conformally Einstein. In fact,  in \cite{leistner-nurowski08} we gave explicit examples of Bach flat pp-waves that are not conformally Einstein.
 \end{remark}

When determining the ambient metric for a gpp-wave for which  the metric $\mathbf{G}$ is Ricci-flat,  we can apply Theorem \ref{fgwalkerthm} and Proposition \ref{liehprop}. Moreover, since all the $\e_a=\del_a$ are parallel, the curvature terms $R_{aijk}$ vanish, but also the $\Theta^\aa$'s are parallel.   We obtain
\begin{corollary}
Let $\mathbf{G}=G_{AB}\der x^A\der x^B$ be a Ricci-flat metric on $\mathbb{R}^{n-2p}$ and $H_{\aa \bb}$ functions of $(n-p)$ variables  $(x^A,x^\bb)$  that define the  gpp-wave
\[\gg
\
=
\
2\delta_{\aa b}  \der x^\aa \der x^b +  H_{\aa \bb}\der x^\aa \der x^\bb+ G_{AB}\der x^A\der x^B \]
on $\mathbb{R}^n$. 
Then an ambient metric for $[\gg]$ is given by
$
\widetilde{\gg}=2\der t\der(\rho t)+t^2(\gg+\h(\rho))
$, where 
$\h= h_{\bb\dd}\der x^\bb \der x^\dd$ and 
whose components satisfy
$\del_a(h_{\bb\dd})=0$ and 
\begin{equation}
\label{fggppansatz}
\begin{array}{rcl} 
2\rho \ddot h_{\bb\dd}+(2-n)\dot hh_{\bb\dd}
-\Delta_\mathbf{G}(h_{\bb\dd})-\Delta_\mathbf{G}(H_{\bb\dd})
&=&O(\rho^m),
\end{array}
\end{equation}
with \oe\ and where
$\Delta_\mathbf{G}$ is the Laplacian of $\mathbf{G}$.
%
%
%
%
%
\end{corollary}

This corollary shows  that in order to obtain Ricci-flat ambient metrics,  for a function $H=H(x^{p+1}, \ldots, x^n)$ we have to solve the equation
\begin{equation}
\label{fggpp1}
\begin{array}{rcl} 
2\rho \ddot h+(2-n)\dot h
-\Delta_\mathbf{G}(h)-\Delta_\mathbf{G}(H)
&=&0,
\end{array}\end{equation}
for a function $h=h(\rho, x^{p+1},\ldots,x^n)$.
This can be solved by standard power series expansion, noticing that its indicial exponents are $s=0$ and $s=n/2$. We extend our results in \cite{AndersonLeistnerNurowski15, leistner-nurowski08},  by the following more general existence statement for gpp-waves.

\begin{theorem}\label{pftheo}
Let $\mathbf{G}$ be a semi-Riemannian metric on $\mathbb{R}^{n-2p}$. Then the following functions  $h=h(\rho, x^{p+1}, \ldots, x^{n})$ are solutions to  equation \eqref{fggpp1} with $h(\rho)\to 0$  when $\rho\downarrow 0$:

 When $n$ is odd:
\be
h=\sum_{k=1}^{\infty} \frac{\Delta_\mathbf{G}^kH}{k!\prod_{i=1}^k(2i-n)}\rho^k
+
\rho^{n/2}\Big(\alpha+\sum_{k=1}^\infty \frac{\Delta_\mathbf{G}^k\alpha}{k!\prod_{i=1}^k(2i+n)}\rho^k\Big),\label{pro}\ee
where $\alpha=\alpha(x^{p+1}, \ldots, x^{n})$ is an arbitrary function of its variables.
In particular, if $\alpha\equiv 0$ this gives an  analytic in $\rho$ solution in a neighbourhood of $\rho=0$ with $h(0)=0$. 

When $n=2s$ is even:
\begin{eqnarray}
\label{pre}
h&=&
\sum_{k=1}^{s-1}
 \frac{\Delta_\mathbf{G}^kH}{k!\prod_{i=1}^k(2i-n)}\rho^k
 +\rho^{s}\Big(\alpha+\sum_{k=1}^\infty \frac{\Delta_\mathbf{G}^k\alpha}{k!\prod_{i=1}^k(2i+n)}\rho^k\Big)
  \\
\nonumber &&+\
c_n\rho^{s}\left(
 \sum_{k=0}^\infty
\left(\log(\rho)
- q_k\right)\frac{\Delta_\mathbf{G}^{s+k}H}{k!\prod_{i=1}^k(2i+n)}\rho^k\right),
\end{eqnarray}
where $\alpha=\alpha(x^{p+1}, \ldots, x^{n})$ and $q_0=q_0(x^{n-p+1}, \ldots , x^n)$ and
\[
q_k(x^{n-p+1}, \ldots , x^n) := q_0(x^{n-p+1}, \ldots , x^n)+\sum_{i=1}^k\frac{n+4i}{i(n+2i)},\]
 for  $k=1, 2, \ldots $,
are arbitrary functions of their variables and the constant $c_n$ is given as follows
\[ c_n:=-
\frac{1}{
(s-1)! \prod_{i=0}^{s-1}(2i-n)}.\]
In particular,  when $\Delta_\mathbf{G}^{s}H\equiv 0$ there are solutions  that are analytic in $\rho$ in a neighbourhood of $\rho=0$ and with $h(0)=0$. These solutions are parametrized by the functions $\alpha$.

\end{theorem}

\bprf
That the given function satisfy equation \eqref{fggpp1} can be checked directly. 
In the case $n$ odd it follows  from Lemma \ref{sol-lemma}.
For $n$ even, the situation is a bit more subtle. We give the formulas for each term, ignoring the term $(\rho^{\frac{n}{2}}(\alpha+\alpha_+))$, for which we have seen that it is in the kernel of $\mathcal D_-$:
\begin{eqnarray*}
\mathcal D_-\left(\sum_{k=1}^{s-1}
 \frac{\Delta_\mathbf{G}^kH}{k!\prod_{i=1}^k(2i-n)}\rho^k
\right)&=&
\Delta_\mathbf{G} H - \frac{\Delta_\mathbf{G}^sH}{(s-1)! \Pi_{i=1}^{s-1}(2i-n)}\rho^{s-1},
\\
\mathcal D_- \left( \rho^s\Delta_\mathbf{G}^s \left(\log(\rho)(H+H_+)\right)\right)
&=&
n\rho^{s-1}\Delta^sH +\frac{n+4}{n+2} \rho^s \Delta_\mathbf{G}^{s+1}H
\\
&&{}
+ \sum_{k=1}^\infty \frac{(n+4(k+1))}{(k+1)! \Pi_{i=1}^{k+1}(2i+n) } \Delta^{s+k+1}H\, \rho^{s+k}
\\
\mathcal D_- \left( \rho^s\Delta_\mathbf{G}^s 
\sum_{k=0}^\infty 
q_k\frac{\Delta_\mathbf{G}^kH}{k!\prod_{i=1}^k(2i-n)}\rho^k\right)
&=&(q_1-q_0)
\rho^s\Delta_\mathbf{G}^{s+1}H,
\\
&&{}
+ \sum_{k=1}^\infty \frac{(q_{k+1}-q_k)(n+2(k+1))}{k!\Pi_{i=1}^{k+1}(2i+n)} \Delta_\mathbf{G}^{s+k+1}H\,\rho^{s+k}.
\end{eqnarray*}
Looking at the $\rho^{s-1}$-terms in these formulas we determine $c_n$ as in the theorem by
\[
- \frac{1}{(s-1)! \Pi_{i=1}^{s-1}(2i-n)}
+nc_n =0.\]
Moreover, looking at  the $\rho^s$-terms, we determine $q_1$ by
\[
\frac{n+4}{n+2}
-(q_1-q_0)=0\]
as given in the theorem, 
and finally the other $q_k$'s by
\[
n+4(k+1) - (q_{k+1}-q_k)(n+2(k+1))(k+1)=0.
\]
This proves the theorem.
\eprf
Summarising, we obtain 
\begin{corollary}
 Let \[\gg
\
=
\
2 \der x^\aa( \delta_{\aa b} \der x^b +  H_{\aa \bb}\der x^\bb)+ G_{AB}\der x^A\der x^B 
\] be a gpp-wave with Ricci-flat metric $\mathbf{G}=G_{AB}\der x^A \der x^B$. 
Then  ambient metrics in the sense of 
Definition \ref{ambientdef} 
for the conformal class $[\gg]$ are
 \beqn
 \widetilde{\gg}&=&2\der(\rho t)\der t+t^2\gg+
 \\
 &&
 +t^2\left(
\left( \sum_{k=1}^{m} \frac{\Delta_\mathbf{G}^k(H_{\aa\bb})}{k!\prod_{i=1}^k(2i-n)}\rho^k
+
\rho^{n/2}\Big(F_{\aa\bb}+\sum_{k=1}^\infty \frac{\Delta_\mathbf{G}^k(F_{\aa\bb})}{k!\prod_{i=1}^k(2i+n)}\rho^k\Big)
\right)\der x^\aa \der x^\bb
 \right)
\eeqn
in which $m=\infty$ when $n$ is odd and $m=\frac{n-2}{2}$ when $n$ is even, and 
 $F_{\aa\cc}=F_{\cc\aa}$ are arbitrary functions on $M$, with $\del_a(F_{\aa\cc})=0$.
 Moreover, 
 \begin{enumerate}
 \item 
When  $n$ is odd,
 $F_{\aa\cc}\equiv 0$ gives the unique analytic Ricci-flat Fefferman-Graham ambient metric. 
 \item 
When $n$ is even and 
$\Delta_{\mathbf G}^{\frac{n}{2}}(H_{\aa\cc})=0$, then the metric $\widetilde{\gg}$ is Ricci-flat.
 
 \item
 When $n$ is even and $\Delta_{\mathbf G}^{\frac{n}{2}}(H_{\aa\cc})\not=0$, then
Ricci-flat but non analytic ambient metrics are given  by formula \eqref{pre}.
\end{enumerate}
\end{corollary}

\subsection{Ambient metrics for Lorentzian pp-waves}
%
Finally we consider Lorentzian pp-waves, i.e., gpp-waves with $p=1$ and  $G_{AB}=\delta_{AB}$. Since $p=1$ we use a different convention as names for the variables: we replace coordinates   $x^1,$ $x^A$, $A=2, \ldots , n-2$, and $x^n$ by
 $v:=x^1$, $y^i=x^{i+1}$, $i=1, \ldots, n-2$, and $u=x^n$. 
We have seen solutions of  equation
~\eqref{fggpp1}
 in Theorem \ref{pftheo}. 
For Lorentzian pp-waves
these are all of the solutions. Here $\Delta_\mathbf{G}=\Delta$ is just the flat Laplacian and we can use the Fourier transform to transform equation \eqref{fggpp1} into an ODE. In fact, in
\cite{AndersonLeistnerNurowski15} we proved the following
\begin{theorem}[\cite{AndersonLeistnerNurowski15}]\label{pptheo}
Let $\Delta$ be the flat Laplacian in $(n-2)$ dimensions.

When $n$ is odd, the most general solutions $h$ to equation (\ref{fggpp1}) with $h(\rho)\to 0$  when $\rho\downarrow 0$ are given by formula (\ref{pro}) in Theorem \ref{pftheo} and parametrized by arbitrary functions
 $\alpha=\alpha(x^{1},\ldots, x^{n-2},u)$.
In particular, there is a unique solution that is analytic in $\rho$ in a neighbourhood of $\rho=0$ with $h(0)=0$. This solution is given  by 
 $\alpha\equiv 0$.

When $n=2s$ is even, the most general solutions $h$ to equation (\ref{fggpp1}) with $h(\rho)\to 0$  when $\rho\downarrow 0$ are given by
\begin{eqnarray}
\label{preq}
h&=&
\sum_{k=1}^{s-1}
 \frac{\Delta^kH}{k!\prod_{i=1}^k(2i-n)}\rho^k
 + 
 \rho^{s}\Big(\alpha+\sum_{k=1}^\infty \frac{\Delta^k\alpha}{k!\prod_{i=1}^k(2i+n)}\rho^k\Big)
\\
\nonumber &&
+\
c_n\rho^{s}
 \sum_{k=0}^\infty \frac{1}{k! \prod_{i=1}^k(2i+n)}\left(
\left(\log(\rho)
- q_k\right){\Delta^{s+k}H}
+ Q\ast 
{\Delta^{s+k}H}\right)\rho^k ,
\nonumber
\end{eqnarray}
where $\alpha=\alpha(y^i,u)$ and $Q=Q(x^i,u)$ are arbitrary functions of their variables, $\ast$ denotes the convolution of two functions with respect to the $y^i$-variables, $c_n$ is the constant defined in Theorem \ref{pftheo}, and the other 
constants are given as follows
\[ q_0 := 0,\ q_k := \sum_{i=1}^k\frac{n+4i}{i(n+2i)},\ \text{  for  $k=1, 2, \ldots $.}\]
In particular, only when $\Delta^{s}H\equiv 0$ there are solutions  that are analytic in $\rho$ in a neighbourhood of $\rho=0$ and with $h(0)=0$. These solutions are not unique but parametrized by the functions $\alpha$.
\end{theorem}

With the results of Corollary \ref{introcorollary}, in particular with the formula for the obstructiont tensor, for Lorentzian pp-waves  we get the complete picture  in Theorem~\ref{pptheo-intro}:
\begin{corollary}
Let \begin{equation}\label{gpp1}
\gg=2\der u\der v +H\,\der u^2 +\sum_{i=1}^{n-2}(\der y^i)^2
\end{equation}
be a Lorentzian pp-wave metric with
 $H=H(y^1, \ldots, y^{n-2},u)$ a function not depending on $v$. Let $\Delta$ is the flat Laplacian in $n-2$ dimensions.
 \begin{enumerate}
 \item
If $n$ is odd, the unique  Ricci-flat ambient metric that is analytic in $\rho$ is 
\[
 \widetilde{\gg}=2\der(\rho t)\der t+t^2\gg+
 +t^2 \left(\sum_{k=1}^{\infty} \frac{\Delta^k(H)}{k!\prod_{i=1}^k(2i-n)}\rho^k\right)\der u^2.
 \]
 Moreover, all non-analytic solutions are 
  parametrized by arbitrary functions
 $\alpha=\alpha(y^{1},\ldots, y^{n-2},u)$ and given by
 formula (\ref{pro}) 
 in Theorem \ref{pftheo}, 
 in which $\Delta_\mathbf{G}$ is replaced by the flat Laplacian.
 \item If $n=2s$ is even the obstruction tensor for $[\gg]$ is a constant multiple of $\Delta^{n/2}(H) \der u^2$. If it vanishes, all  Ricci-flat ambient metrics that are analytic in $\rho$ are given by
 \[
 \widetilde{\gg}=2\der(\rho t)\der t+t^2\gg
 +t^2\left(
  \sum_{k=1}^{s-1} \frac{\Delta^k(H)}{k!\prod_{i=1}^k(2i-n)}\rho^k
+
\sum_{k=0}^\infty \frac{\Delta^k(\alpha)}{k!\prod_{i=1}^k(2i+n)}\rho^{\frac{n}{2}+k}
\right)\der u^2,
 \]
 where $\alpha=\alpha(y^1, \ldots, y^{n-2},u)$ is an arbitrary smooth function.
Independently of the vanishing of the obstruction tensor,   non-analytic ambient metrics  can be obtained from 
formula (\ref{preq}) in Theorem \ref{pptheo}.
  \end{enumerate}
 
 \end{corollary}

\bibliography{GEOBIB}
\end{document}